\documentclass[11pt, ngerman, english]{article}

\usepackage{hyperref}
\usepackage{babel}
\usepackage{translations}
\usepackage[utf8]{inputenc}
\usepackage{amsmath}
\usepackage{amsthm}
\usepackage{amssymb}
\usepackage{amsfonts}
\usepackage{graphicx}
\usepackage{enumitem}
\usepackage{verbatim}
\usepackage{float}
\usepackage{chngcntr}
\usepackage{geometry}
\usepackage{color}
\usepackage{centernot}
\usepackage[toc,page]{appendix}
\usepackage{cite}
\usepackage{maths}

\newtheorem{defin}{Definition}[section]
\newtheorem{lem}[defin]{Lemma}
\newtheorem{theo}[defin]{Theorem}
\newtheorem{rem}[defin]{Remark}
\newtheorem{prop}[defin]{Proposition}
\newtheorem{ex}[defin]{Example}
\newtheorem{corol}[defin]{Corollary}
\newtheorem{prob}[defin]{Problem}

\DeclareMathOperator{\diag}{diag}
\DeclareMathOperator{\spn}{span}
\DeclareMathOperator{\ran}{range}

\renewcommand{\conj}{\lconj}

\newcounter{typeC}[subsection]

\newcommand{\newType}{\stepcounter{typeC}\subsubsection*{Type \arabic{subsection}.\Roman{typeC}}\label{type:\arabic{subsection}.\Roman{typeC}}}
\newcommand{\Href}[1]{\hyperref[type:#1]{Type #1}}
\newcommand{\Pref}[1]{\hyperref[p:#1]{Type #1}}
\newcommand{\newTypeProof}{\stepcounter{typeC}\subsubsection*{\Href{\arabic{subsection}.\Roman{typeC}}}\label{p:\arabic{subsection}.\Roman{typeC}}}

\renewcommand{\iff}{if and only if }
\makeatletter
\newcommand*{\rom}[1]{\expandafter\@slowromancap\romannumeral #1@}
\makeatother

\begin{document}

\begin{titlepage}

%\begin{addmargin}[-1cm]{-3cm}
\begin{center}
\large

\hfill
\vfill

\begingroup
\huge On Polar Decompositions with Commuting Factors in Indefinite Inner Product Spaces \\ \bigskip % Thesis title
\endgroup

Julian Kern % Your name

\vfill

\includegraphics[width=6cm]{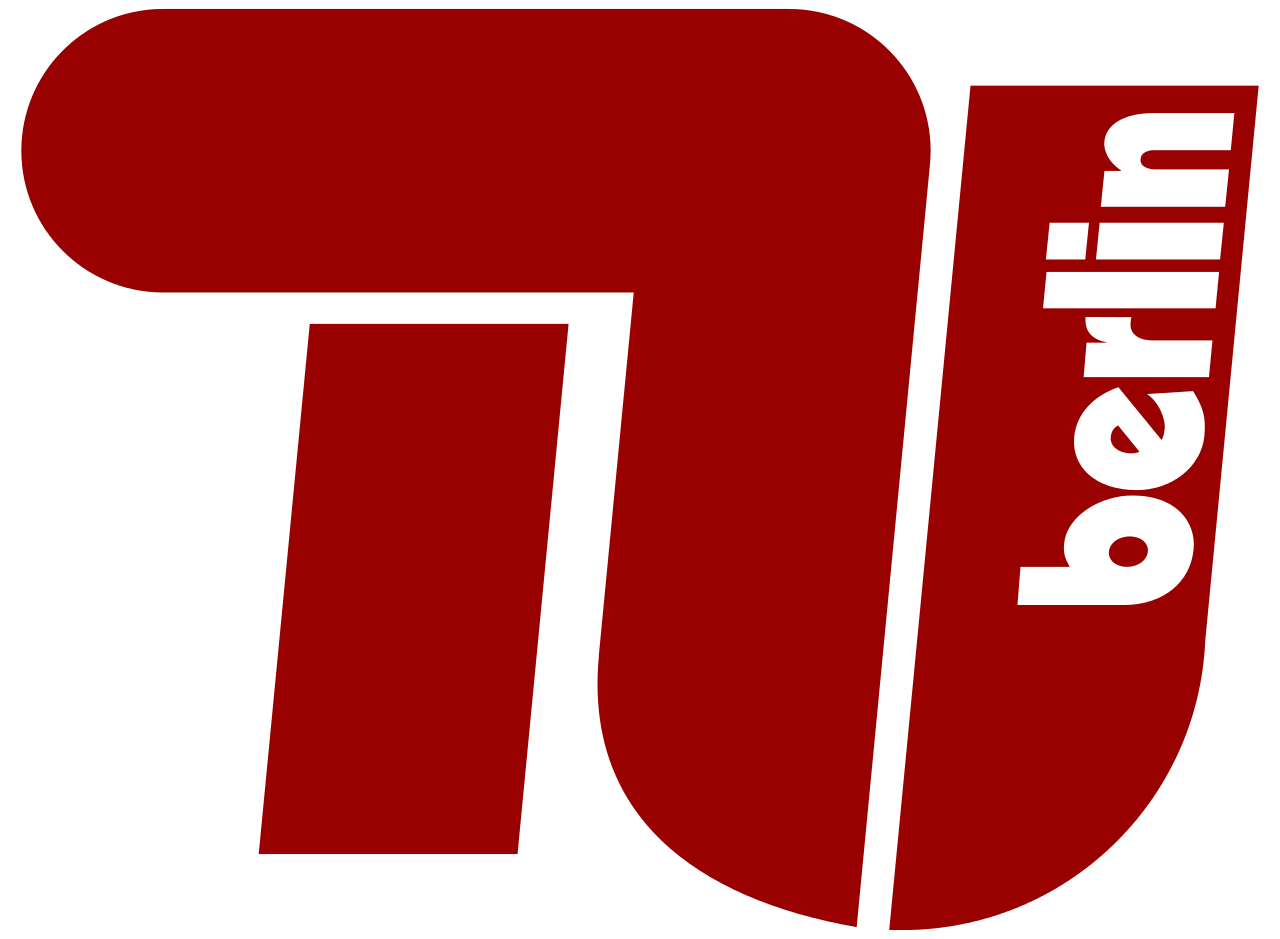} \\ \medskip % Picture

%\mySubtitle \\ \medskip % Thesis subtitle
%\myDegree \\
%\myDepartment \\
%\myFaculty \\
%\myUni \\
FG Numerische Mathematik\\
Fakultät II -- Mathematik und Naturwissenschaften\\
Technische Universität Berlin\\
\bigskip
Zur Erlangung des akademischen Grades \\
\bigskip
\textbf{Bachelor of Science}
\bigskip \\
Gutachter:\\
Prof. Christian Mehl \\
Dr. Michael Karow \\
\bigskip
Betreuer:\\
Prof. Christian Mehl

\bigskip

Juni, 2018 % Time and version

\vfill

\end{center}
%\end{addmargin}

\end{titlepage}

%\maketitle

\newpage\hfill\thispagestyle{empty}
\newpage

\pagenumbering{roman}
\section*{Eidesstattliche Erklärung}

Hiermit erkläre ich, dass ich die vorliegende Arbeit selbstständig und eigenhändig sowie ohne unerlaubte fremde Hilfe und ausschließlich unter Verwendung der aufgeführten Quellen und Hilfsmittel angefertigt habe.
\bigskip
 
\noindent\textit{Berlin, den}

\smallskip
\begin{center}
\begin{tabular}{ c }
\\ \hline
Julian Kern \\
\end{tabular}
\end{center}

\cleardoublepage\hfill\newpage

\begin{abstract}
This thesis examins a generalisation of polar decompositions to indefinite inner product spaces. The necessary general theory is studied and some general results are given. The main part of the thesis focuses on polar decompositions with commuting factors: First, a proof for a generalisation of the link between polar decomposition with commuting factors and normal matrices is given. Then, some properties of such decompositions are studied and it is shown that the commutativity of the factors only depends on the selfadjoint part. Eventually, polar decompositions with commuting factors are studied under similarity transformations that do not alter the structure of the space. For this purpose, normal forms are decomposed and analysed.
\end{abstract}
\begin{verbatim}





\end{verbatim}

\selectlanguage{ngerman}
\begin{abstract}
Die vorliegende Bachelorarbeit behandelt eine mögliche Verallgemeinerung der Polarzerlegung auf Räume, die statt eines (positiv definiten) Skalarprodukts lediglich ein indefinites Produkt besitzen. Die notwendige allgemeine Theorie wird im ersten Abschnitt der Arbeit hergeleitet. Im Hauptteil der Arbeit wird ein Beweis zur Verallgemeinerung der Korrespondenz zwischen Polarzerlegungen mit kommutierenden Faktoren und normalen Matrizen zusammengetragen. Weiter werden einige Eigenschaften solcher Zerlegungen untersucht und insbesondere nachgewiesen, dass die Kommutativität der Faktoren ausschließlich von dem selbstadjungierten Anteil abhängt. Schließlich werden Polarzerlegungen mit kommutierenden Faktoren unter Ähnlichkeitstransformationen betrachtet, die die Struktur des zugrundeliegenden Raumes nicht verändern. In diesem Rahmen werden Normalformen zerlegt und analysiert.
\end{abstract}

\selectlanguage{english}
\cleardoublepage

\tableofcontents

\newpage
\pagenumbering{arabic}

\section{Introduction}

This bachelor thesis discusses  polar decompositions of matrices in indefinite inner product spaces with special interest in polar decompositions with commuting factors.

Every complex number $z = re^{i\varphi}$ can be decomposed into a scaling factor $r\geq 0$ and a rotation through an angle $\varphi\in(-\pi,\pi]$. If $z\neq 0$, this decomposition is unique with $r > 0$, and if $z$ is real, the decomposition is real too, i.e. $\varphi\in\llbrace 0,\pi\rrbrace$. This can be generalised to matrices with entries in $\mathbb{F}\in\llbrace\mathbb{R},\mathbb{C}\rrbrace$. In a Euclidean space, every matrix $X\in\mathbb{F}^{n,n}$ can be decomposed into a product $X = UA$ of an orthogonal (resp.~unitary) matrix $U\in\mathbb{F}^{n,n}$ and a symmetric (resp.~Hermitian) positive semidefinite matrix $A\in\mathbb{F}^{n,n}$. If $X$ is invertible, this decomposition is uniquely determined. As for complex numbers, we can interpret the orthogonal (resp.~unitary) matrix as a rotation and the symmetric (resp.~Hermitian) positive-definite matrix as a generalised scaling, since $A = V^*\Lambda V$, where $V$ is a unitary matrix and $\Lambda$ is a diagonal matrix with positive entries. 

Despite the possibility to decompose $X = UA = \widetilde{A}\widetilde{U}$ in two ways into orthogonal (resp.~unitary) $U,\widetilde{U}\in\mathbb{F}^{n,n}$ and symmetric (resp.~Hermitian) positive semidefinite $A,\widetilde{A}\in\mathbb{F}^{n,n}$ factors, there is no guarantee that $U = \widetilde{U}$ and $A = \widetilde{A}$. This means that the commutativity of the factors is lost when passing from $n=1$ to higher dimensions. One can show that the commutativity is obtained for normal matrices only (cf. Theorem \ref{theo:pol_com_eucli}).\\
 
We study a generalisation of the theory of polar decompositions to spaces with \emph{indefinite inner products}, i.e. spaces without the positive-definiteness of the scalar product. This implies that the terms \emph{normal, orthogonal/unitary} etc.~have to be replaced by the corresponding terms \emph{$H$-normal, $H$-unitary etc.~}in indefinite inner product spaces. The fundamental theory of Indefinite Linear Algebra has been developed in \cite{glr}. The theory of polar decompositions in indefinite inner product spaces has been studied extensively in \cite{br,bmrrr1,BMRRR, bmrrr} by Y. Bolshakov et al. Whereas these papers offer a less restrictive definition of polar decomposition, Mackey et al.~propose a more restrictive definition of a \textquotedblleft computable generalised polar decomposition\textquotedblright{} in \cite{mmt}. In this thesis, we will concentrate on polar decompositions with commuting factors. In order to compensate for this restriction, we will use the less restrictive definition. Analogously to the Euclidean case, $H$-normal matrices are crucial in the examination of $H$-polar decompositions with commuting factors. The problem of characterising these $H$-normal matrices has proved to be very difficult. In fact, it has been shown in \cite{gr} that this problem is wild, i.e.~as hard as the problem of classification up to similarity of pairs of commuting matrices. However, the cases of $H$ admitting a small number of negative eigenvalues have been studied in \cite{gr,hs,HS} and normal forms for one and two negative eigenvalues could be derived. This has led to a factorisation of these forms in \cite{lmmr}. A conclusive answer in the case $\mathbb{F} = \mathbb{C}$ has been given in \cite{mrr}, where an equivalent condition for the existence of $H$-polar decompositions with commuting factors was proved.

The search for polar decompositions in indefinite inner product spaces is motivated by linear optics, where polar decompositions in the 4-dimensional Minkowski space are computed as a necessary and sufficient condition on a matrix to be a feasible transformation (cf.. \cite[Section 6]{BMRRR}). Another application is the solution of generalised Procrustes problems as they occur in multidimensional scaling. These methods are used in psychology to compare a test person's opinion on the similarities and dissimilarities of a finite number of given objects (cf. \cite{k,m}).\\ 

This thesis is divided into four sections. Following the introduction, the second section discusses preliminaries including basic theory on polar decompositions in indefinite inner product spaces. The main section studies $H$-polar decompositions with commuting factors. After having discussed the main result from \cite{mrr}, we will investigate how to characterise the factors of the polar decompositions. The last section will consider some yet unsolved issues regarding polar decompositions (with commuting factors) in indefinite inner product spaces. The remainder of this thesis contains a list of $H$-polar decompositions with commuting (or semicommuting) factors of indecomposable normal matrices. The thesis is mainly based on \cite{mrr} and \cite{lmmr}.

\section{Preliminaries}
In this section, we recall the theory of polar decompositions in Euclidean spaces and introduce a generalisation of this theory to indefinite inner product spaces.

Throughout the thesis, let $\mathbb{F}$ denote the field $\mathbb{R}$ of real numbers or the field $\mathbb{C}$ of complex numbers and let $I_n\in\mathbb{F}^{n,n}$ be the identity matrix. Furthermore, let $\langle\cdot,\cdot\rangle$ denote the standard Euclidean product on $\mathbb{F}^n$, given by
\[
\langle x,y\rangle = x^*y = \sum_{i=1}^n \lconj{x_i}y_i,
\]
where $x = (x_1,\dots,x_n),y = (y_1,\dots, y_n)\in\mathbb{F}^n$ and $x^* = \lconj{x}^T$ is the complex conjugate transpose of $x$.

\subsection{Polar decomposition in Euclidean spaces}

Since all results from this section are well known, only the two main results will be proved for completeness.

\begin{defin}
Let $A\in\mathbb{F}^{n,n}$.
\begin{enumerate}
	\item The uniquely determined matrix $A^* = \lconj{A}^T$, such that
	\[
	\langle Ax, y\rangle = \langle x, A^*y\rangle
	\]
	for all $x,y\in\mathbb{F}^{n}$, is called the \emph{adjoint matrix} of $A$.
	
	\item We call $A$ \emph{selfadjoint} if $A^* = A$.
	\item We call $A$ \emph{unitary} if $A^*A = I_n$.
	\item We call $A$ \emph{normal} if $A^*A = AA^*$, i.e. if $A$ commutes with its adjoint.
\end{enumerate}
\end{defin}

\begin{rem}
\normalfont
A matrix $A\in\mathbb{F}^{n,n}$ is unitary if and only if it preserves the inner product, i.e. 
\[
\langle x, A^*Ay\rangle = \langle Ax, Ay\rangle = \langle x, y\rangle.
\]
\end{rem}

\begin{theo}
A selfadjoint matrix has only real eigenvalues.
\end{theo}
\begin{theo}[Unitary diagonalisation]\label{theo:unit_diag}
A matrix is normal \iff it is unitarily diagonalisable.
\end{theo}
\begin{theo}[Simultaneous diagonalisation]
Two diagonalisable matrices $A,B\in\mathbb{F}^{n,n}$ commute \iff they are simultaneously diagonalisable, i.e. there exists a unitary $U\in\mathbb{F}^{n,n}$ such that $U^*AU$ and $U^*BU$ are diagonal.
\end{theo}

\begin{theo}[Singular value decomposition]
Let $A\in\mathbb{F}^{n,n}$.~Then there exists unitary $V,W\in\mathbb{F}^{n,n}$ and a uniquely determined $\Sigma = \diag(\sigma_1,\dots,\sigma_r,0\dots,0)\in\mathbb{R}^{n,n}$ with ${1\leq r\leq n}$ and $\sigma_1\geq \cdots\geq\sigma_r > 0$ such that
\[
A = V\Sigma W^*. 
\]
\end{theo}
\begin{corol}[Polar decomposition]
Let $A\in\mathbb{F}^{n,n}$.~Then there exists an unitary ${U\in\mathbb{F}^{n,n}}$ and a uniquely determined selfadjoint positive-semidefinite ${S\in\mathbb{F}^{n,n}}$ such that
\[
A = US.
\]
If $A$ is nonsingular, $U$ is uniquely determined and $S$ is positive-definite.
\end{corol}
\begin{proof}
In this thesis we prove only the existence of the polar decomposition. The uniqueness of the decomposition follows from \cite[Theorem 1.29]{h}. Let ${A = V\Sigma W^*}$ be a singular value decomposition of $A$ and define
\[
U := VW^*\quad\text{ and }\quad S = W\Sigma W^*.
\]
Then $U$ is unitary as a product of unitary matrices and $S$ is positive-semidefinite and selfadjoint since it is unitarily similar to $\Sigma$. Furthermore, it is
\[
US = (VW^*)(W\Sigma W^*) = V\Sigma W^* = A.\qedhere
\]
\end{proof}
\begin{rem}\normalfont
The above decomposition is also called a \emph{right} polar decomposition. One may deduce from the proof the decomposition of $A$ into
\[
A = S'U
\]
with $S' := USU^*$ into a \emph{left} polar decomposition. Since one can easily be converted into the other (and thus share the same properties), we can concentrate on the right polar decomposition.
\end{rem}

We are now able to state and prove a theorem that characterises polar decompositions with commuting factors in Euclidean spaces. It is relevant for us, since it motivates the main result in section \ref{sec:commuting_factors}.

\begin{theo}\label{theo:pol_com_eucli}
Let $A\in\mathbb{F}^{n,n}$. The following assertions are equivalent.
\begin{enumerate}[label=(\roman*)]
	\item $A$ is normal.
	\item There exists a polar decomposition of $A$ with commuting factors.
	\item There exists an unitary matrix $V\in\mathbb{F}^{n,n}$ such that $A = VA^*$.
\end{enumerate}
\end{theo}
\begin{proof}
\begin{description}
	\item[\textquotedblleft (i) $\Rightarrow$ (ii)\textquotedblright.] Let $A$ be normal. Then there exists a unitary $W\in\mathbb{F}^{n,n}$ such that
	\[
	W^*AW = \diag(d_1,\dots, d_k, 0\dots, 0) =: D
	\]
	is diagonal with $d_i\neq 0$, $i = 1,\dots,k$ and $1\leq k\leq n$. Hence, $d_j = r_je^{i\varphi_j}$, $j=1,\dots, k$, with $r_1,\dots, r_k > 0$ and $\varphi_1,\dots,\varphi_k\in(-\pi,\pi]$. Thus, we obtain a polar decomposition
	\[
	D = \underbrace{\diag(e^{i\varphi_1}, \dots, e^{i\varphi_k}, 1,\dots, 1)}_{=:U'}\cdot\underbrace{\diag(r_1,\dots, r_k,0,\dots, 0)}_{=:S'}
	\]
	of $D$ which yields a polar decomposition
	\[
	A = WDW^* = \underbrace{(WU'W^*)}_{=:U}\cdot\underbrace{(WS'W^*)}_{=:S}
	\]
	of $A$. Since $U$ and $S$ are, by definition, simultaneously diagonalisable, they commute.
	
	\item[\textquotedblleft (ii) $\Rightarrow$ (iii)\textquotedblright.] Let $A = US$ be a polar decomposition of $A$ with commuting factors. This gives
	\[
	A = US = U^2\cdot (U^*S^*) = U^2\cdot (SU)^* = U^2\cdot (US)^* = U^2\cdot A^*.
	\]
	Thus, (ii) follows with $V := U^2$.
	\item[\textquotedblleft (iii) $\Rightarrow$ (i)\textquotedblright.] Let $A = VA^*$ for some unitary matrix $V$. Then
	\[
	AA^* = (AV^*)(VA^*) = (VA^*)^*A = A^*A.
	\]
	It follows that $A$ is normal.\qedhere
\end{description}
\end{proof}

\subsection{Fundamental concepts of Indefinite Linear Algebra}

In this section, fundamental elements of Indefinite Linear Algebra are introduced. For this, we follow partially \cite[Chapter 2.2]{glr}.

\begin{defin}[Indefinite inner product space]
A bilinear (resp.~sesquilinear) form $[\cdot,\cdot]:\mathbb{F}^{n}\times\mathbb{F}^n \rightarrow\mathbb{F}$ is called an \emph{indefinite inner product} on $\mathbb{F}^n$ if it is symmetric (resp.~Hermitian) and nondegenerate. In this case, we call $(\mathbb{F}^n, [\cdot,\cdot])$ an \emph{indefinite inner product space}.
\end{defin}

We can identify an indefinite inner product with an uniquely determined nonsingular selfadjoint matrix $H\in\mathbb{F}^{n,n}$ by the relation
\[
[x,y] = \langle Hx, y\rangle.
\]
We will thus use these two notations interchangeably. 

Throughout the section we will assume $H$ to be a nonsingular and selfadjoint matrix inducing the indefinite inner product $[\cdot,\cdot]$.

\begin{defin}
Let $W\subseteq\mathbb{F}^n$ be a subspace of $\mathbb{F}^n$.
\begin{enumerate}
	\item We say that $W$ is \emph{$H$-degenerate} if there exists $w\in W\setminus\llbrace 0\rrbrace$ such that
	\[
	\forall v\in W,\quad [v,w] = 0,
	\]
	i.e. if the bilinear (resp. sequilinear) form $[\cdot,\cdot]\vert_W$ is degenerate. Otherwise, we say that $W$ is \emph{$H$-nondegenerate}.
	
	\item Let $x\in\mathbb{F}^n$. We say that $x$ is \emph{$H$-positive} resp.~\emph{$H$-negative} resp.~\emph{$H$-neutral} if ${[x,x] > 0}$ resp.~$[x,x] < 0$ resp.~$[x,x] = 0$.
	
	\item We say that $W$ is \emph{$H$-positive} if every vector $x\in W$ is $H$-positive. If there exists no other $H$-positive subspace $W\subsetneq V \subseteq \mathbb{F}^n$, we say that $W$ is a \emph{maximal $H$-positive subspace}. We define (maximal) $H$-negative and $H$-neutral subspaces in an analogous manner.
	
	\item A system $(e_1, \dots, e_k)\in\left(\mathbb{F}^n\right)^k$ is called \emph{$H$-unitary} if $[e_i,e_i]\in\llbrace -1,1\rrbrace$ and $[e_i,e_j] = 0$ for $i,j=1,\dots,k$, $i\neq j$.
	
	\item We call
	\[
	W^{[\perp]} := \llbrace v\in\mathbb{F}^n\;\vert\; \forall w\in W,\; [v,w] = 0\rrbrace
	\]
	the \emph{$H$-unitary complement} of $W$.
\end{enumerate}
\end{defin}

The following results show that indefinite inner products behave in some respect {analogously} to the Euclidean inner product.

\begin{theo}\label{theo:h_unit_basis_exist}
There exists an $H$-unitary basis of $\mathbb{F}^n$.
\end{theo}
\begin{proof}
Since $H$ is selfadjoint, all its eigenvalues are real and there exists a unitary basis $\mathcal{B} := (e_1',\dots,e_n')$ of $\mathbb{F}^n$ consisting of eigenvectors of $H$. Let $(p^+, p^-, 0)$ denote the signature of $H$ and assume without loss of generality that the eigenvalues $(\lambda_1,\dots,\lambda_n)\in\mathbb{R}^n$ corresponding to $\mathcal{B}$ are sorted in descending order. This implies $\lambda_i > 0$ for $i=1,\dots,p^+$ and $\lambda_i < 0$ for $i=p^++1,\dots, n$. Now define a new basis $(e_1,\dots, e_n)$ by
\[
e_i := \dfrac{1}{\sqrt{\vert \lambda_i\vert}}\cdot e_i'\;,
\]
giving
\[
[e_i, e_i] = \dfrac{1}{\vert \lambda_i\vert}\cdot e_i'^*He_i' = \dfrac{\lambda_i}{\vert\lambda_i\vert} \underbrace{\langle e_i', e_i'\rangle}_{=1} = \llbrace\begin{array}{ll}
1 & \text{, for } 1\leq i\leq p^+\\
-1&\text{, for } p^+ < i\leq n
\end{array}\right.
\]
and analogously $[e_i, e_j] = 0$ for $i\neq j$. We conclude that $(e_1,\dots, e_n)$ is an $H$-unitary basis of $\mathbb{F}^n$.
\end{proof}

\begin{lem}\label{lem:unit_unabh}
Every $H$-unitary system is linearly independent.
\end{lem}
\begin{proof}
Let $(e_1,\dots,e_k)$ be an $H$-unitary system. Furthermore let $\lambda_1,\dots,\lambda_k\in\mathbb{F}$ be such that $\sum_{i=1}^ k \lambda_ie_i = 0$. This implies for $j=1,\dots,k$ in particular
\[
0 = \left[\sum_{i=1}^ k \lambda_ie_i, e_j\right] = \sum_{i=1}^ k \lconj{\lambda_i} [e_i, e_j] = \lconj{\lambda_j}[e_j,e_j]\in \llbrace -\lconj{\lambda_j}, \lconj{\lambda_j}\rrbrace,
\]
i.e. $\lambda_j = 0$ for $j=1,\dots,k$.
\end{proof}
\begin{lem}
Let $W\subseteq\mathbb{F}^ n$ be a subspace of $\mathbb{F}^ n$. Then the following statements hold.
\begin{enumerate}
	\item If $W\neq \llbrace 0\rrbrace$ is neutral, then $W$ is degenerate.
	\item $W^ {[\perp]}$ is a subspace of $\mathbb{F}^ n$.
	\item $W$ is degenerate \iff $W^ {[\perp]}$ is degenerate.
	\item It is
	\[
	W^ {[\perp]} = H^ {-1}W^ {\perp}.
	\]
	\item If $W$ is nondegenerate, it holds
	\[
	\mathbb{F}^n = W \oplus W^{[\perp]}.
	\]
\end{enumerate}
\end{lem}
\begin{proof}
\begin{enumerate}
	\item Suppose that $W$ is neutral, i.e. $[x,x] = 0$ for every $x\in W$. This implies
	\[
	[x,y] = \dfrac{1}{2}\cdot ([x+y,x+y] - [x,x] - [y,y]) = 0
	\]
	for all $x,y\in W$. We conclude that $W$ is degenerate.
	\item The assertion follows directly from the sesquilinearity of the indefinite inner product.
	\item Let $W$ be degenerate. Then there exists a $w\in W\setminus\llbrace 0\rrbrace$ such that
	\[
	\forall v\in W,\quad [v,w] = 0.
	\]
	This implies $w\in W^{[\perp]}$. But since
	\[
	\forall v'\in W^{[\perp]},\quad [v', w] = 0
	\]
	by definition of $W^{[\perp]}$, this gives that $W^{[\perp]}$ is degenerate. The other implication follows analogously.
	
	\item Let $w\in W^{\perp}$. Then
	\[
	[v, H^{-1}w] = v^*HH^{-1}w = v^*w = \langle v,w\rangle = 0
	\]
	for all $v\in W$. It follows that $H^{-1}w\in W^{[\perp]}$ and thus $H^{-1}W^\perp\subseteq W^{[\perp]}$. Now let $w\in W^{[\perp]}$ and $z = Hw$. Then
	\[
	\langle v,z\rangle = v^*z = v^*Hw = [v, w] = 0
	\]
	for all $v\in W$ and thus $z\in W^\perp$ which implies $W^{[\perp]}\subseteq H^{-1}W^\perp$.
	
	\item From the fact that $W$ (and thus $W^{[\perp]}$) are nondegenerate follows $W\cap W^{[\perp]} = \llbrace 0\rrbrace$. From 4. it follows that $\dim W^{[\perp]} = \dim W^\perp$ and thus
	\[
	\dim W + \dim W^{[\perp]} = \dim W + \dim W^\perp = n.\qedhere
	\]
\end{enumerate}
\end{proof}

\begin{theo}[$H$-unitary basis extension]\label{theo:basis_ext}
If $(e_1,\dots, e_k)$ is an $H$-unitary system, there exist vectors $e_{k+1},\dots, e_n\in\mathbb{F}^n$ such that $(e_1,\dots, e_n)$ is an $H$-unitary basis.
\end{theo}
\begin{proof}
It suffices to show that one can add a vector to an $H$-unitary system consisting of $k < n$ vectors, since it follows from Lemma \ref{lem:unit_unabh}, that the system is linearly independent. The assertion then follows by an inductive argument. 

Consider the subspace
\[
W := \spn(e_1,\dots,e_k)
\]
and let $v := \sum_{i=1}^ k \lambda_ie_i\in W\setminus\llbrace 0\rrbrace$, assuming without loss of generality that $\lambda_1\neq 0$. This gives $[v, e_1] = \lconj{\lambda_1}[e_1,e_1]\neq 0$ and we conclude that $W$ is nondegenerate.

Since $W$ is nondegenerate, $W^{[\perp]}$ is nondegenerate and thus cannot be neutral. Hence, there exists a vector $x\in W^{[\perp]}$ such that $[x,x] \neq 0$. Define
\[
e_{k+1} := \dfrac{1}{\sqrt{\vert [x,x]\vert}}\cdot x,
\]
giving $[e_{k+1},e_{k+1}]\in\llbrace -1,1\rrbrace$. Furthermore, $(e_1,\dots, e_{k+1})$ is an $H$-unitary system.
\end{proof}

In order to discuss structured decomposition as polar decompositions, a generalisation of the concept of selfadjoint, unitary and normal matrices is necessary. Furthermore, we need to analyse which properties can be translated into our context of indefinite inner product spaces.

\begin{defin}
Let $A\in\mathbb{F}^{n,n}$.
\begin{enumerate}
	\item The uniquely determined matrix $A^{[*]} := A^{[*]_H} := H^{-1}A^*H$ such that
	\[
	[Ax, y] = [x, A^{[*]}y]
	\]
	for all $x,y\in\mathbb{F}^{n}$ is called the \emph{$H$-adjoint matrix} of $A$.
	
	\item We say that $A$ is \emph{$H$-selfadjoint}, if $A^{[*]} = A$.
	
	\item We say that $A$ is \emph{$H$-unitary}, if $A^{[*]}A = I_n$.
	
	\item We say that $A$ is \emph{$H$-normal}, if $A^{[*]}A  = AA^{[*]}$, i.e. if $A$ commutes with its $H$-adjoint.
\end{enumerate}
\end{defin}

\begin{rem}
\normalfont
\begin{enumerate}
	\item The existence and uniqueness of the $H$-adjoint can be proved as follows. First, we note that for all $x,y\in\mathbb{F}^n$
	\[
	[Ax, y] = \langle HAx, y\rangle = \langle HAH^{-1}\cdot Hx, y\rangle = \langle Hx, H^{-1}A^*Hy\rangle = [x, H^{-1}A^*Hy].
	\]
	Furthermore, suppose that $B$ and $B'$ are two $H$-adjoint matrices of $A$. This implies
	\[
	[x, (B-B')y] = [Ax, y] - [Ax, y] = 0
	\]
	for all $x,y\in\mathbb{F}^n$. Since the indefinite inner product is nondegenerate, it follows that $B = B'$.
	
	\item As in the Euclidean case, a matrix $A\in\mathbb{F}^{n,n}$ is $H$-unitary \iff it preserves the inner product, i.e.
	\[
	[Ax, Ay] = [x,y]
	\]	
	for all $x,y\in\mathbb{F}^n$.
	
	\item The above definition can be reformulated as
	\begin{itemize}
		\item $A$ is $H$-selfadjoint $\Leftrightarrow$ $HA = A^*H$.
		\item $A$ is $H$-unitary $\Leftrightarrow$ $H = A^*HA$.
		\item $A$ is $H$-normal $\Leftrightarrow$ $AH^{-1}A^*H = H^{-1}A^*HA$.
	\end{itemize}
	
	\item One can easily verify that
	\[
	(AB)^{[*]} = B^{[*]}A^{[*]}\quad\text{ and }\quad \left(A^{-1}\right)^{[*]} = \left(A^{[*]}\right)^{-1} =: A^{-[*]}.
	\]
	
	\item As in the definite case, every $H$-selfadjoint or $H$-unitary matrix is $H$-normal. However, we lose the characterisation from Theorem \ref{theo:unit_diag}: There are $H$-normal matrices that are not diagonalisable. More precisely, already $H$-selfadjoint are not necessarily diagonalisable (cf. Example \ref{ex:h-normal_diag}).
	
	\item Other properties are lost partially, e.g. if $A$ is $H$-selfadjoint, then it can have complex eigenvalues, but they are symmetric with respect to the real axis. This means that if $\lambda\in\mathbb{C}$ is an eigenvalue of $A$, then $\lconj{\lambda}$ is also an eigenvalue of $A$. Moreover, the algebraic multiplicity of $\lambda$ and $\lconj{\lambda}$ coincide (cf. Example \ref{ex:h-normal_diag}). 
\end{enumerate}
\end{rem}

\begin{ex}\normalfont\label{ex:h-normal_diag}
Consider the nonsingular and selfadjoint matrix
\[
H = \begin{bmatrix}
0 & 0 & 1\\
0 & 1 & 0\\
1 & 0 & 0
\end{bmatrix}.
\]
In this case, the $H$-adjoint of a matrix $A$ is the \textquotedblleft reflection of $\lconj{A}$ over its antidiagonal\textquotedblright. Thus, the matrices
\[
X_1 = \begin{bmatrix}
0 & 1 & 0\\
0 & 0 & 1\\
0 & 0 & 0
\end{bmatrix}\quad\text{ and }\quad X_2 = \begin{bmatrix}
i & 0 & 0\\
0 & 0 & 0\\
0 & 0 & -i
\end{bmatrix}
\]
are both $H$-selfadjoint. Furthermore, $X_1$ is $H$-normal but not diagonalisable and $X_2$ has complex eigenvalues. Nevertheless, the spectrum of $X_2$ is symmetric with respect to the real line.
\end{ex}
\begin{ex}\normalfont\label{ex:h-unit_base}
With the matrix $H$ from the example above,
\[
X_3 = \begin{bmatrix}
2 & 0 & 0\\
0 & 1 & 0\\
0 & 0 & 1/2
\end{bmatrix}
\]
is $H$-unitary. Since $$\left[\begin{bmatrix}
2\\0\\0
\end{bmatrix},\begin{bmatrix}
2\\0\\0
\end{bmatrix}\right] = 0,$$ we cannot conclude that the columns of an $H$-unitary matrix form an $H$-unitary basis.
\end{ex}

The Example \ref{ex:h-unit_base} shows that we cannot characterise $H$-unitary matrices directly by an $H$-unitary basis. We thus need a new characterisation of $H$-unitary matrices which will be provided by the next two theorems.

\begin{theo}
A matrix $U\in\mathbb{F}^{n,n}$ is $H$-unitary \iff there exist two $H$-unitary bases $(e_1,\dots, e_n)$ and $(e_1',\dots, e_n')$ such that
\[
Ue_i = e_i'\quad\text{ and }\quad [e_i, e_i] = [e_i', e_i']
\]
for $i=1,\dots, n$.
\end{theo}
\begin{proof} 
First, suppose $U$ to be $H$-unitary and let $(e_1,\dots, e_n)$ be an arbitrary $H$-unitary basis. Since $U$ preserves the indefinite inner product, it suffices to take $(Ue_1,\dots, Ue_n)$ as the second basis. For the other implication, consider a matrix $U$ such that
\[
Ue_i = e_i'\quad\text{ and }\quad [e_i, e_i] = [e_i', e_i']
\]
for two $H$-unitary bases $(e_1,\dots, e_n)$, $(e_1',\dots, e_n')$ and $i=1,\dots,n$. Now consider arbitrary $x = \sum_{i=1}^n \lambda_ie_i,y = \sum_{i=1}^n \mu_ie_i\in\mathbb{F}^n$. Then
\begin{align*}
[Ux, Uy] &= \sum_{i=1}^n\sum_{j=1}^n \lconj{\lambda_i}\mu_j[Ue_i, Ue_j]\\
&= \sum_{i=1}^n\sum_{j=1}^n \lconj{\lambda_i}\mu_j[e_i', e_j']\\
&= \sum_{i=1}^n \lconj{\lambda_i}\mu_i [e_i', e_i']\\
&= \sum_{i=1}^n \lconj{\lambda_i}\mu_i [e_i, e_i]\\
&= \sum_{i=1}^n\sum_{j=1}^n \lconj{\lambda_i}\mu_j [e_i, e_j] = [x,y].
\end{align*}
Hence, $U$ preserves the indefinite inner product.
\end{proof}

This theorem is incomplete in the sense that we cannot assume that there is an $H$-unitary matrix translating arbitrary $H$-unitary bases into one another: We still need to show that two $H$-unitary bases have the same number of $H$-positive and $H$-negative vectors. 

\begin{theo}[Sylvester's law of inertia, cf. \cite{sb}]
There exists an $H$-positive subspace $P\subseteq\mathbb{F}^n$ and an $H$-negative subspace $N\subseteq\mathbb{F}^n$ such that
\[
\mathbb{F}^n = P\oplus N.
\]
If $(p^+, p^-, 0)$ is the signature of $H$, it holds
\[
p^+ = \dim P\quad\text{ and }\quad p^- = \dim N
\]
for any such decomposition.
\end{theo} 
\begin{theo}
Let $(p^+, p^-, 0)$ be the signature of $H$. Then every $H$-unitary basis consists of $p^+$ $H$-positive and $p^-$ $H$-negative vectors.
\end{theo}
\begin{proof}
Let $(e_1,\dots, e_n)$ be an $H$-unitary basis of $\mathbb{F}^n$ such that $e_1,\dots, e_k$ are $H$-positive and $e_{k+1},\dots, e_n$ are $H$-negative for some $1\leq k\leq n$. Consider an arbitrary $v = \sum_{i=1}^n \lambda_ie_i\in\mathbb{F}^n$. Then
\[
[v,v] = \sum_{i=1}^n\sum_{j=1}^n \lconj{\lambda_i}\lambda_j [e_i, e_j] = \sum_{i=1}^n \vert \lambda_i\vert^2 [e_i, e_i]
\]
and therefore it follows that 
\[
P:=\spn(e_1,\dots, e_k)\quad\text{ and }\quad N:=\spn(e_{k+1},\dots, e_n)
\]
are respectively an $H$-positive and an $H$-negative subspace. Since $P\cap N=\llbrace0\rrbrace$ and $k + (n-k) = n$ imply
\[
\mathbb{F}^n = P\oplus N,
\]
we conclude that $k = p^ +$.
\end{proof}

\subsection{Polar decompositions in indefinite inner product spaces}\label{ssec:pol_decomp_indef}

In this section, we will discuss one possible way of generalising the polar decomposition to indefinite inner product spaces. This thesis adopts the definition used by I. Gohberg, P. Lancaster and L. Rodman.

As in the previous section, we consider a nonsingular and selfadjoint matrix $H\in\mathbb{F}^{n,n}$ and the induced indefinite inner product $[\cdot,\cdot]$ on $\mathbb{F}^n$.

\begin{defin}[$H$-polar decomposition]
Let $X\in\mathbb{F}^{n,n}$. A decomposition
\[
X = UA
\]
into matrices $U,A\in\mathbb{F}^{n,n}$ is called \emph{$H$-polar decomposition} of $X$ if $U$ is $H$-unitary and $A$ is $H$-selfadjoint.
\end{defin}

This definition can be refined by adding more constraints to $A$ as it is done in Euclidean spaces. One problem is that positive-(semi)definiteness cannot be generalised in a unique way. D. S. Mackey, N. Mackey and F. Tisseur propose in \cite{mmt} to restrict the definition to matrices $A$ with eigenvalues $\Lambda(A) \subseteq \llbrace z\in\mathbb{C}\;\vert\; \Re(z) > 0\rrbrace$. This definition of \textquotedblleft computable\textquotedblright{} polar decompositions may facilitate the actual computation of such decompositions as it is explained in \cite[Section 6]{mmt}, but it omits singular matrices in which we are also interested in.

As in the Euclidean case, it is possible to define the \emph{left} $H$-polar decomposition. But as one can easily pass from one to the other, it suffices to study the \emph{right} $H$-polar decomposition that we defined above.\\

We will now discuss equivalent conditions for the existence of an $H$-polar decomposition. The following two theorems are preliminaries and as for the proofs, we will mainly follow \cite[Section 8]{br} and \cite[Section 2]{bmrrr}.
\begin{theo}[Witt]\label{theo:witt}
Let $V,W\subseteq\mathbb{F}^{n}$ be two subspaces of $\mathbb{F}^n$. If $U_0:V\rightarrow W$ is an isomorphism which preserves the indefinite inner product, i.e. which satisfies $${[U_0(x), U_0(y)] = [x,y]}$$ for all $x,y\in V$, there exists an $H$-unitary matrix $U\in\mathbb{F}^{n,n}$ such that
\[
\forall x\in V,\quad Ux = U_0(x).
\]
\end{theo}
\begin{proof}
Since $U_0$ is an isomorphism between $V$ and $W$, we can define $m:= \dim V = \dim W$. Denote by $(m_+, m_-, m_0)$ the signature of $[\cdot,\cdot]\vert_V$ and let $(e_1,\dots, e_m)$ be a basis of $V$ such that
\[
[e_i, e_j] = \llbrace\begin{array}{ll}
1&\text{ for } i = j = m_0+1,\dots, m_0+m_+\\
-1&\text{ for } i = j = m_0+m_++1,\dots, m\\
0&\text{ otherwise }
\end{array}\right..
\]
The existence of such a basis may be proved as we did for $H$-unitary bases in Theorem \ref{theo:h_unit_basis_exist}. Since $(e_1, \dots, e_m)$ is a basis, the system of linear equations
\begin{equation}\label{eq:witt}
[e_i, x_j] = \delta_{i,j},\qquad i,j = 1,\dots, m
\end{equation}
admits a solution $\left(\widetilde{e}_1,\dots,\widetilde{e}_m\right)$. Without loss of generality, we can assume that $\widetilde{e}_k$ is neutral for all $k=1,\dots, m_0$: otherwise replace $\widetilde{e}_k$ by $\widetilde{e}_k - \dfrac{1}{2}\left[\widetilde{e}_k, \widetilde{e}_k\right]e_k$. The resulting system still solves \eqref{eq:witt} and 
\[
\left[\widetilde{e}_k - \dfrac{1}{2}\left[\widetilde{e}_k, \widetilde{e}_k\right]e_k, \widetilde{e}_k - \dfrac{1}{2}\left[\widetilde{e}_k, \widetilde{e}_k\right]e_k\right] = \left[\widetilde{e}_k, \widetilde{e}_k\right] - \left[\widetilde{e}_k,  \widetilde{e}_k\right]\cdot\underbrace{\left[e_k, \widetilde{e}_k\right]}_{=1} + \dfrac{1}{4}\left[\widetilde{e}_k, \widetilde{e}_k\right]\cdot\underbrace{[e_k, e_k]}_{=0} = 0.
\]
Now define
\[
e_k' := \dfrac{1}{\sqrt{2}}\left(e_k - \widetilde{e}_k\right)\quad\text{ and }\quad e_k'' := \dfrac{1}{\sqrt{2}}\left(e_k + \widetilde{e}_k\right)
\]
for $k=1,\dots,m_0$. Then
\begin{align*}
[e_k', e_k'] &= \dfrac{1}{2}[e_k, e_k] - [e_k, \widetilde{e}_k] + \dfrac{1}{2}[\widetilde{e}_k, \widetilde{e}_k] = -[e_k,\widetilde{e}_k] = -1\,,
\end{align*}
\begin{align*}
[e_k'', e_k''] = \dfrac{1}{2}[e_k, e_k] + [e_k, \widetilde{e}_k] + \dfrac{1}{2}[\widetilde{e}_k, \widetilde{e}_k] = [e_k, \widetilde{e}_k] = 1
\end{align*}
and
\[
[e_k', e_k''] = \dfrac{1}{2}[e_k, e_k] + \dfrac{1}{2}[e_k, \widetilde{e}_k] - \dfrac{1}{2}[\widetilde{e}_k, e_k] - \dfrac{1}{2}[\widetilde{e}_k,\widetilde{e}_k] = 0\,.
\]
We conclude that $(e_{m_0+1},\dots, e_m, e_1',\dots, e_{m_0}', e_1'',\dots, e_{m_0}'')$ is an $H$-unitary system which can be extended with Theorem \ref{theo:basis_ext} to an $H$-unitary basis $\mathcal{B}$. Let $e_l$ denote the extending vectors, ${l=2m_0+m_++m_-+1,\dots,n}$.

Now consider $f_k := U_0(e_k)$ for $k=1,\dots, m$ and the resulting basis $(f_1,\dots, f_m)$ of $W$. This basis has the same properties as the basis $(e_1,\dots,e_m)$ in $V$. Thus, we define $f_k'$ and $f_k''$ for $k=1,\dots, m_0$ and $f_l$ for $l=2m_0+m_++m_-+1,\dots,n$ in an analogous way. Let $\mathcal{B}'$ denote the resulting $H$-unitary basis.

Since the bases $\mathcal{B}$ and $\mathcal{B}'$ are $H$-unitary, we can assume without loss of generality that the matrix $U\in\mathbb{F}^{n,n}$ defined by
\begin{align*}
\forall k=1,\dots, m_0,\quad& Ue_k' = f_k',\quad Ue_k'' = f_k''\\
\forall k=m_0+1,\dots,m,\quad& Ue_k = f_k\\
\forall l = 2m_0 + m_+ + m_- + 1,\dots, n, \quad& Ue_l = f_l
\end{align*}
is $H$-unitary. In fact, it suffices to sort the extending vectors $e_l$ and $f_l$ such that the $H$-positive vectors come first. Furthermore,
\[
Ue_k = \sqrt{2}(Ue_k' + Ue_k'') = \sqrt{2}(f_k' + f_k'') = f_k
\]
for all $k=1,\dots, m_0$ and thus $Ux = U_0(x)$ for all $x\in V$.
\end{proof}

With Witt's theorem on $H$-unitary extensions of inner product preserving isomorphisms, we can find an equivalent condition on the existence of an $H$-polar decomposition. However, we will first prove a more general result, which will be needed later for the main result.

\begin{theo}\label{theo:unit_trans_gen}
Let $X,Y\in\mathbb{F}^{n,n}$. Then there exists an $H$-unitary matrix $U\in\mathbb{F}^{n,n}$ such that $X = UY$ \iff
\[
X^{[*]}X = Y^{[*]}Y\quad\text{ and }\quad \ker X = \ker Y.
\]
\end{theo}
\begin{proof}[Proof of Theorem \ref{theo:unit_trans_gen}]
First let $X = UY$ for some $H$-unitary $U$. Then
\[
X^{[*]}X = Y^{[*]}U^{[*]}UY = Y^{[*]}Y.
\]
Furthermore, $U$ is nonsingular and thus $\ker X = \ker Y$.

Now consider matrices $X,Y\in\mathbb{F}^{n,n}$ such that $X^{[*]}X = Y^{[*]}Y$ and $\ker X = \ker Y$. Choose a basis $(f_1,\dots, f_m)$ of $\ran Y$ such that
\[
[f_i, f_j] = \llbrace \begin{array}{lll}
1&\text{, if }& i = j = 1,\dots, m_+\\
-1&\text{, if }& i = j = m_+ + 1,\dots, m_+ + m_-\\
0&\text{, otherwise } &
\end{array}\right..
\]
where $(m_+,m_-,m_0)$ denotes the signature of $[\cdot,\cdot]\vert_{\ran(Y)}$ and $m = m_++m_-+m_0$. (Such a basis can be constructed as in the proof of Theorem \ref{theo:h_unit_basis_exist}.) Since $(f_1,\dots,f_m)$ forms a basis of $\ran Y$, there exist vectors $e_1,\dots,e_m$ such that $f_i = Ye_i$, $i=1,\dots,m$. Define $g_i := Xe_i$, $i=1,\dots,m$. We show that $(g_1,\dots,g_m)$ forms a basis of $\ran X$. Note first that $\ker X = \ker Y$ implies $\dim\ran X = \dim\ran Y$. Hence, it suffices to prove that $(g_1,\dots, g_m)$ is linearly independent. Consider arbitrary $\lambda_1,\dots,\lambda_m$ such that
\[
\sum_{i=1}^m \lambda_ig_i = 0.
\]
This means that $\sum_{i=1}^m \lambda_ie_i\in\ker X = \ker Y$ and thus
\[
\sum_{i=1}^m \lambda_if_i = 0.
\]
Since $(f_1,\dots, f_m)$ is linearly independent, it follows that $\lambda_1 = \cdots = \lambda_m = 0$. Furthermore, it holds
\[
[g_i,g_i] = [Xe_i, Xe_i] = [e_i, X^{[*]}Xe_i] = [e_i, Y^{[*]}Ye_i] = [Ye_i, Ye_i] = [f_i,f_i]
\]
for all $i=1,\dots,m$. Therefore, the linear map $U_0:\ran Y\rightarrow\ran X$, defined by $U_0(f_i) = g_i$ for $i=1,\dots,m$ is an isomorphism preserving the indefinite inner product. Hence, by Theorem \ref{theo:witt}, There exists an $H$-unitary matrix $U\in\mathbb{F}^ {n,n}$ such that $Ux = U_0(x)$ for all $x\in\ran Y$. It remains to prove that $X = UY$. Let $v\in\mathbb{F}^n$ be arbitrary and let $\lambda_1,\dots,\lambda_m\in\mathbb{F}$ be such that $Yv = \sum_{i=1}^m \lambda_if_i\in\ran Y$. Then
\[
v - \sum_{i=1}^m \lambda_ie_i\in\ker Y = \ker X
\]
and thus $Xv = \sum_{i=1}^m \lambda_ig_i$. This implies
\[
Xv = \sum_{i=1}^m \lambda_ig_i = \sum_{i=1}^m \lambda_i Uf_i = U\cdot\sum_{i=1}^m \lambda_if_i = UYv.\qedhere
\]
\end{proof}

From this more general result, we immediately deduce an equivalent condition on the existence of $H$-polar decompositions.

\begin{corol}[Existence of an $H$-polar decomposition]
A matrix $X\in\mathbb{F}^{n,n}$ admits an $H$-polar decomposition \iff there exists an $H$-selfadjoint matrix $A$ such that $X^{[*]}X = A^2$ and $\ker X = \ker A$.
\end{corol}
\begin{proof}
It suffices to set $Y = A$. Then $A^{[*]}A = A^2 = X^{[*]}X$ and the above theorem gives the existence of an $H$-unitary matrix $U\in\mathbb{F}^{n,n}$ such that $X = UA$.
\end{proof}

\begin{rem}\normalfont
\begin{enumerate}
	\item Whereas $X^*X = Y^*Y$ already implies $\ker X = \ker Y$ in Euclidean spaces, it is an additional restriction in indefinite inner product spaces. Therefore, this condition appears naturally if one generalises a theorem from Euclidean spaces to indefinite inner product spaces.
	\item Not every matrix $X\in\mathbb{F}^{n,n}$ admits an $H$-polar decomposition. Consider the following example from \cite[Section 8]{br}.
\end{enumerate}
\end{rem}

\begin{ex}
\normalfont
Let\[
H = \begin{bmatrix}
0 & 1\\
1 & 0
\end{bmatrix},\quad X = \begin{bmatrix}
0 & 1\\
0 & 1/2
\end{bmatrix}.
\]
Since 
\[
X^{[*]}X = \begin{bmatrix}
0 & 1\\
0 & 0
\end{bmatrix}
\]
does not admit any square root, $X$ cannot have an $H$-polar decomposition.
\end{ex}

\section{Polar Decompositions with Commuting Factors}\label{sec:commuting_factors}

In this section, we consider $H$-polar decompositions with commuting factors. The main result of this section generalises Theorem \ref{theo:pol_com_eucli} and was first proposed in \cite{mrr}. 

Throughout this section we assume $H\in\mathbb{F}^{n,n}$ to induce the indefinite inner product $[\cdot,\cdot]$.

\subsection{Equivalent conditions to the existence of polar decompositions with commuting factors}

\begin{theo}[{\cite[Theorem 10]{mrr}}]\label{theo:commuting_factors}
Let $\mathbb{F} = \mathbb{C}$ and $X\in\mathbb{C}^{n,n}$. The following assertions are equivalent.
\begin{enumerate}[label=(\roman*)]
	\item $X$ admits an $H$-polar decomposition with commuting factors.
	\item $X$ is $H$-normal and $\ker X = \ker X^{[*]}$.
	\item There exists an $H$-unitary matrix $V$ such that $X = VX^{[*]}$.
\end{enumerate}
\end{theo}

We note that the theorem restricts the equivalence to $\mathbb{F} = \mathbb{C}$. Even though it has been shown for $H\in\mathbb{R}^{n,n}$ in \cite{bmrrr} that a real matrix $X\in\mathbb{R}^{n,n}$ admits an $H$-polar decomposition over $\mathbb{R}$ if and only if it admits a complex one, this equivalence cannot be extended to $H$-polar decompositions with commuting factors.
 
\begin{ex}\normalfont
Consider the setting of \Href{2.IV} from the appendix.
\[
X = \begin{bmatrix}
0 & 1 & r\\
0 & 0 & -1\\
0 & 0 & 0
\end{bmatrix}\quad\text{ and }\quad H = \begin{bmatrix}
0 & 0 & 1\\
0 & 1 & 0\\
1 & 0 & 0
\end{bmatrix},
\]
where $r\in\mathbb{R}$. The matrix $X$ is $H$-normal and $\ker X = \ker X^{[*]}$, but there exists no $H$-polar decomposition of $X$ such that its factors commute. To prove this, suppose that $X = UA$ is an $H$-polar decomposition with commuting factors. Since $\ker A = \ker X$ and $A^{[*]} = A$, it holds that
\[
A = \begin{bmatrix}
0 & a_{12} & a_{13}\\
0 & a_{22} & a_{12}\\
0 & 0 & 0
\end{bmatrix}
\]
for some $a_{12},a_{13},a_{22}\in\mathbb{R}$ such that $a_{12}^2 - a_{22}a_{13}\neq 0$. In section \ref{sec:witt_extensions}, we show that $AX = XA$ is necessary for $UA = AU$. But this means that
\[
\begin{bmatrix}
0 & 0 & -a_{12}\\
0 & 0 & -a_{22}\\
0 & 0 & 0
\end{bmatrix} = AX = XA = \begin{bmatrix}
0 & a_{22} & a_{12}\\
0 & 0 & 0\\
0 & 0 & 0
\end{bmatrix},
\]
and thus $a_{12}^2 - a_{22}a_{13} = 0$. 
\end{ex}
At the end of this section, we will see that we can extend Theorem \ref{theo:commuting_factors} to $\mathbb{F} = \mathbb{R}$ in the case of a restricted class of matrices.

For the proof of Theorem \ref{theo:commuting_factors}, we note that the equivalence (ii) $\Leftrightarrow$ (iii) is a direct consequence of Theorem \ref{theo:unit_trans_gen} for $Y = X^{[*]}$. Thus, we only need to show the equivalence (i) $\Leftrightarrow$ (iii). 

\begin{lem}\label{lem:equiv_AUUA_XUUX}
Let $X = UA\in\mathbb{F}^{n,n}$ be an $H$-polar decomposition. Then $UA = AU$ \iff $UX = XU$.
\end{lem}
\begin{proof}
Suppose that the factors of the decomposition commute. Then
\[
UX = U^2A = UAU = XU.
\]
On the other hand, if $UX = XU$, then
\[
AU = U^{[*]}UAU = U^{[*]}XU = U^{[*]}U X = X = UA.\qedhere
\]
\end{proof}
\begin{lem}\label{lem:h_unit_square_root}
Let $V\in\mathbb{C}^{n,n}$ be an $H$-unitary matrix and $X\in\mathbb{C}^{n,n}$ such that $VX = XV$. Then there exists an $H$-unitary matrix $U\in\mathbb{C}^{n,n}$ such that
\[
V = U^2\quad\text{ and }\quad UX = XU.
\]
\end{lem}
\begin{proof}
Let $\Lambda(A)$ denote the spectrum of a matrix $A\in\mathbb{C}^{n,n}$. Since the square root cannot be defined on the whole complex plane, we need to study two cases.
\begin{description}
	\item[Case 1: $\Lambda(V)\cap (-\infty,0{]}=\emptyset$.] Since $V$ has only finitely many eigenvalues, there exists a curve $\Gamma$ that encloses $\Lambda(V)$ and does not cross $(-\infty,0]$. Now let $f:\mathbb{C}\setminus(-\infty,0]\rightarrow\mathbb{C}$ be the principal branch of the complex square root. Then $f$ is analytic on $\Gamma$ and its enclosed area and it holds (cf. \cite[Theorem 1.12]{h})
	\[
	f(V) = \dfrac{1}{2\pi i}\int_\Gamma f(z)(zI_n - V)^{-1} \d{z}.
	\]
	The fact that $V$ is $H$-unitary implies the equality $VH^{-1} = H^{-1}(V^*)^{-1}$ and thus for every $z\in\mathbb{C}\setminus(-\infty, 0]$
	\begin{align*}
	H(zI_n - V)^{-1} &= \left((zI_n-V)H^{-1}\right)^{-1} \\
	&= \left(H^{-1}\left(zI_n - (V^*)^{-1}\right)\right)^{-1}\\
	&= \left(zI_n - (V^*)^{-1}\right)^{-1}H.
	\end{align*}
	This yields
	\[
	Hf(V) = f\left((V^*)^{-1}\right)H.
	\]
	For this, we assume without loss of generality that $\Gamma$ also encloses $\Lambda\left((V^*)^{-1}\right)$. The results \cite[Theorems 1.15 and 1.18]{h} finally give
	\[
	f\left((V^*)^{-1}\right) = \left(f(V)^*\right)^{-1}.
	\]
	From these two equations it follows that $U := f(V)$ is $H$-unitary. By definition, it satisfies $V = U^2$ and \cite[Theorem 1.13]{h} states that the commutativity is inherited from $V$.
	
	\item[Case 2: $\Lambda(V)\cap (-\infty,0{]}\neq\emptyset$.] Since $V$ has only finitely many eigenvalues, there exists an angle $\theta\in [0,2\pi)$ such that the ray $e^{i\theta} [0,+\infty)$ does not contain any eigenvalues. Now consider the matrix $V' := e^{i(\pi-\theta)}V$. Then
	\[
	V'^{[*]}V' = \left(e^{i(\theta-\pi)}V^{[*]}\right)\left(e^{i(\pi-\theta)}V\right) = V^{[*]}V = I_n
	\]
	and thus $V'$ is $H$-unitary. Furthermore, $V'$ commutes with $X$ and does not have any nonpositive real eigenvalues. According to the first case, there exists an $H$-unitary matrix $U'$ such that $V' = (U')^2$ and $U'$ commutes with $X$. Define the $H$-unitary matrix $U := e^{i(\theta-\pi)/2}U'$. This matrix commutes with $X$ and 
	\[
	U^2 = e^{i(\theta-\pi)}(U')^2 = e^{i(\theta-\pi)}V' = V.\qedhere
	\]
\end{description}
\end{proof}

\begin{proof}[Proof of Theorem \ref{theo:commuting_factors}]
We already noted that the equivalence (ii) $\Leftrightarrow$ (iii) is a corollary of Theorem \ref{theo:unit_trans_gen}. We will now prove the equivalence (i) $\Leftrightarrow$ (iii).

First, suppose that there exists an $H$-polar decomposition $X = UA$ with commuting factors. Then
\[
X = UA = U^ 2(U^ {[*]}A) = U^ 2 (AU)^ {[*]} = U^ 2 (UA)^ {[*]} = U^ 2X^ {[*]}.
\]
It suffices to set $V := U^ 2$.

Now suppose that there exists an $H$-unitary matrix $V$ such that $X = VX^{[*]}$. Then
\[
XV = VX^{[*]}V = V(VX^{[*]})^{[*]}V = VXV^{[*]}V = VX.
\]
By Lemma \ref{lem:h_unit_square_root}, there exists an $H$-unitary square root $U$ of $V$ that commutes with $X$. By defining $A := UX^ {[*]}$, this gives
\[
X = VX^ {[*]} = UA.
\]
Furthermore, $A$ is $H$-selfadjoint as
\[
A^{[*]} = XU^{[*]} = U^2X^{[*]}U^{[*]} = U^2(UX)^{[*]} = U^2 (XU)^{[*]} = U^2 U^{[*]} X^{[*]} = UX^{[*]} = A.
\]
Thus, the decomposition $X = UA$ is an $H$-polar decomposition. Finally, Lemma \ref{lem:equiv_AUUA_XUUX} implies that its factors commute.
\end{proof}

We have seen that Theorem \ref{theo:commuting_factors} cannot be extended to $\mathbb{F} = \mathbb{R}$ for all matrices. However, if we restrict ourselves to a restricted class of matrices, at least a necessary condition for the existence of an $H$-polar decomposition with commuting factors can be proved. Also, this result relates the study of $H$-polar decomposition with commuting factors to the more restrictive definition we shortly discussed at the beginning of Section \ref{ssec:pol_decomp_indef}. As for the proofs, we mainly rely on \cite{hmt}.

\begin{theo}\label{theo:comm_fac_real_case}
Let $\mathbb{F} = \mathbb{R}$ and let $X\in\mathbb{R}^{n,n}$ be $H$-normal.Then $X$ admits an $H$-polar decomposition with commuting factors, if the following conditions are satisfied.
\begin{itemize}
	\item The matrix $X^{[*]}X$ has no negative eigenvalue.
	\item If $0$ is an eigenvalue of $X^{[*]}X$, then it is simple.
	\item It holds $\ker X^{[*]}X = \ker X$.
\end{itemize}
\end{theo}

We note that we do not need the additional kernel constraint $\ker X^{[*]} = \ker X$. With Theorem \ref{theo:commuting_factors}, this follows from the existence of an $H$-polar decomposition with commuting factors. Theorem \ref{theo:comm_fac_real_case} relates $H$-polar decompositions with commuting factors to the more restrictive definition of $H$-polar decompositions. In fact, the tree latter conditions of Theorem \ref{theo:comm_fac_real_case} are equivalent to the existence of the so-called \emph{canonical generalised polar decomposition}. This decomposition represents a generalisation of the \emph{generalised polar decomposition}, that we discussed at the beginning of section \ref{ssec:pol_decomp_indef}, to singular matrices. In this case, only the $H$-selfadjoint factor is uniquely determined.

In the following, we assume $H\in\mathbb{R}^{n,n}$ to be symmetric nonsingular.

\begin{theo}[Sylvester equation, cf. \cite{s}]\label{theo:sylv_eq}
Let $A,B\in\mathbb{C}^{n,n}$. Then the following assertions are equivalent.
\begin{enumerate}[label=(\roman*)]
	\item $A$ and $B$ have no common eigenvalue.
	\item The equation $AX = XB$ has the unique solution $X = 0$.
\end{enumerate}
\end{theo}

\begin{lem}\label{lem:decomposition_selfadjoint}
Let $A\in\mathbb{R}^{n,n}$ be $H$-selfadjoint. Then there exists a nonsingular matrix $T\in \mathbb{R}^{n,n}$ such that
\[
T^{-1}AT = \begin{bmatrix}
A_1 & 0\\
0 & A_0
\end{bmatrix}\quad\text{ and }\quad T^*HT = \begin{bmatrix}
H_1 & 0\\
0 & H_0
\end{bmatrix},
\]
where $A_1$ and $H_1$ have the same size, $A_1$ is nonsingular and $A_0$ is nilpotent.
\end{lem}
\begin{proof}
Let $T$ be such that
\[
T^{-1}AT = \begin{bmatrix}
A_1 & 0\\
0 & A_0
\end{bmatrix}
\]
is in real Jordan form with $A_1$ nonsingular and $A_0$ nilpotent. Furthermore, let
\[
\begin{bmatrix}
H_1 & H_2\\
H_2^* & H_0
\end{bmatrix} := T^*HT
\]
such that $H_1$ has the same size as $A_1$. Since $A = A^{[*]} = H^{-1}A^*H$, we obtain
\[
\begin{bmatrix}
A_1^*H_1 & A_1^*H_2\\
A_0^*H_2^* & A_0^*H_0
\end{bmatrix} = T^*A^*HT = T^*HAT = \begin{bmatrix}
H_1H_1 & H_2A_0\\
H_2^*H_1 & H_0A_0
\end{bmatrix}.
\]
This yields $A_1^*H_2 = H_2A_0$, and since $A_0$ is nilpoltent and $A_1$ is nonsingular, Theorem \ref{theo:sylv_eq} implies $H_2 = 0$.
\end{proof}

\begin{rem}\normalfont
Consider $X\in\mathbb{R}^{n,n}$. The previous lemma implies the existence of a nonsingular matrix $T\in\mathbb{R}^{n,n}$ such that
\[
T^{-1}X^{[*]}XT = \begin{bmatrix}
B_1 & 0\\0&B_0
\end{bmatrix}\quad\text{ and }\quad T^*HT = \begin{bmatrix}
H_1 & 0\\0 & H_0
\end{bmatrix},
\]
where $B_1$ and $H_1$, $B_1$ is nonsingular and $B_0$ is nilpotent. Setting $Y := T^{-1}XT$ and $H':= T^*HT$, we then get
\[
Y^{[*]_{H'}}Y = \begin{bmatrix}
B_1 & 0\\0 & B_0
\end{bmatrix}.
\]
Furthermore, we show in Section \ref{ssec:indecomp} that $X$ admits an $H$-polar decomposition with commuting factors if and only if $Y$ admits an $H'$-polar decomposition with commuting factors. 
\end{rem}

\begin{proof}[Proof of Theorem \ref{theo:comm_fac_real_case}]
Let $X\in\mathbb{R}^{n,n}$ be $H$-normal such that $X^{[*]}X$ has no negative eigenvalues, if $0$ is an eigenvalue of $X$ then it is semisimple and $\ker(X^{[*]}X) = \ker X$. From the previous remark, we can assume that
\[
X^{[*]}X = \begin{bmatrix}
B_1 & 0\\0 & B_0
\end{bmatrix}\quad\text{ and }\quad H = \begin{bmatrix}
H_1 & 0\\0 & H_0
\end{bmatrix},
\]
where $B_1$ and $H_1$ have the same size, $B_1$ is nonsingular and $B_0$ is nilpotent. Since $0$ is (at most) a semisimple eigenvalue of $X^{[*]}X$, we have $B_0 = 0$. The kernel equality $\ker(X^{[*]}X) = \ker X$ then yields
\[
X = \begin{bmatrix}
X_1 & 0\\X_2 & 0
\end{bmatrix},
\]
where $X_1$ has the same size as $B_1$. Furthermore, the matrix $\begin{bmatrix}
X_1\\X_2
\end{bmatrix}$ has full rank. Since $B_1$ has no nonpositive eigenvalue, we can define the nonsingular matrix $A_1 := B_1^{1/2}$. Since $X^{[*]}X$ is $H$-self-adjoint, $B_1$ is necessarily $H_1$-selfadjoint and thus
\[
A_1^{[*]_{H_1}} = H_1^{-1}\left(B_1^{1/2}\right)^*H_1 = H_1^{-1} \left(B_1^*\right)^{1/2}H_1 = \left(H_1^{-1}B_1^*H_1\right)^{1/2} = B_1^{1/2}.
\]
We conclude that $A_1$ is also $H_1$-selfadjoint. Define
\[
A := \begin{bmatrix}
A_1 &0\\0&0
\end{bmatrix}\quad\text{ and }\quad U_0 := \begin{bmatrix}
X_1A_1^{-1} & 0\\
X_2A_1^{-1} & 0
\end{bmatrix}.
\]
Apparently, $A$ is $H$-selfadjoint and $A^2 = X^{[*]}X$. Furthermore, $X = U_0A$ and
\[
[U_0Ax, U_0Ay] = [Xx, Xy] = [x, X^{[*]}Xy] = [x,A^2y] = [Ax, Ay]
\]
for all $x,y\in\mathbb{R}^n$. We conclude that $U_0:\ran A\rightarrow\ran X$ is an isomorphism preserving the indefinite inner product. Witt's Theorem \ref{theo:witt} then implies the existence of an $H$-unitary $U\in\mathbb{R}^{n,n}$ such that $Ux = U_0x$ for all $x\in\ran A$. Thus, $X = UA$ is an $H$-polar decomposition.

It remains to show that $X = UA$ has commuting factors. Since $X$ is $H$-normal, we have
\[
\begin{bmatrix}
H_1X_1^*H_1X_1 & 0\\0&0
\end{bmatrix} = \begin{bmatrix}
B_1 &0\\0&0
\end{bmatrix} = X^{[*]}X = XX^{[*]} = \begin{bmatrix}
X_1H_1^{-1}X_1^*H_1 & X_1H_1^{-1}X_2^*H_0\\
X_2H_1^{-1}X_1^*H_1 & X_2H_1^{-1}X_2^*H_0
\end{bmatrix}.
\]
This gives
\[
X_2H_1^{-1}\times [X_1^*, X_2^*] = 0
\]
and thus $X_2 = 0$, since $[X_1^*, X_2^*]$ has full rank and $H_1$ and $H_2$ are nonsingular. Furthermore,
\[
H_1^{-1}X_1^*H_1X_1 = B_1 = X_1H_1^{-1}X_1^*H_1
\]
and thus
\[
X_1B_1 = X_1H_1^{-1}X_1^*H_1X_1 = B_1 X_1.
\]
Since $A_1$ is a matrix function of $B_1$, this yields $X_1A_1 = A_1X_1$. Moreover,
\[
\begin{bmatrix}
X_1 & 0\\0 & 0
\end{bmatrix} = X = UA = U\begin{bmatrix}
A_1 &0\\0 & 0
\end{bmatrix}
\]
and thus
\[
U = \begin{bmatrix}
X_1A^{-1} & U_1\\0 & U_2
\end{bmatrix}.
\]
As $X_1$ and $A_1$ are nonsingular, the fact that $U$ is $H$-unitary implies $U_1 = 0$ and thus
\[
AU = \begin{bmatrix}
A_1X_1A_1^{-1} & 0\\0 & 0
\end{bmatrix} = \begin{bmatrix}
X_1A_1A_1^{-1} &0\\0&0
\end{bmatrix} = X.
\]
\end{proof}

Higham et al. prove in \cite{hmt} that the $H$-selfadjoint factor in the constructed $H$-polar decomposition is uniquely determined by the property of having all its eigenvalues in the open right half plane (with 0). If $X$ is nonsingular, then the $H$-unitary factor is also uniquely determined. This special $H$-polar decomposition is called \emph{generalised polar decomposition}.

The necessary and sufficient conditions for the existence of a canonical generalised polar decomposition is given by the three latter conditions in Theorem \ref{theo:comm_fac_real_case}: The matrix $X$ has to be such that $X^{[*]}X$ has no negative eigenvalues, if $0$ is an eigenvalue then it has to be simple, and $\ker (X^{[*]}X) = \ker X$. This is also true in the case $\mathbb{F} = \mathbb{C}$ (the proof is completely analogous).

The generalised polar decomposition can be approximated numerically by different recursive formulas. This implies that we can approximate a (special) $H$-polar decomposition with commuting factors.

\begin{theo}[cf. \cite{hmt}]\label{theo:approx}
Let $X\in\mathbb{F}^{n,n}$ be nonsingular and $H$-normal. If $X^{[*]}X$ has no negative eigenvalue, then the iteration defined by
\[
U_0 := X\quad\text{ and }\quad U_{k+1} := 2U_k(I_n + U_k^{[*]}U_k)^{-1}
\]
converges quadratically to the $H$-unitary factor of the generalised polar decomposition which has commuting factors.
\end{theo}

\subsection{Witt extensions and families of polar decompositions with commuting factors}\label{sec:witt_extensions}

Lemma \ref{lem:equiv_AUUA_XUUX} shows that an $H$-polar decomposition $X = UA$ has commuting factors if and only if $X$ and $U$ commute. This equivalence does not hold for the pair $X$ and $A$. In fact, we have in general the implications
\[
UX = XU \quad\Leftrightarrow\quad AU = UA \quad \begin{array}{ll}
\Rightarrow\\
\centernot\Leftarrow
\end{array}\quad AX = XA.
\]
The implication $AU = UA \Rightarrow AX = XA$ follows directly from
\[
AX = AUA = UAA = XA
\]
and the following two examples from \cite{mrr} show that the equivalence does not hold in general. Note that the first example illustrates the fact that $H$-polar decompositions with commuting factors do coexist with $H$-polar decompositions satisfying the weaker condition $AX = XA$, whereas the second shows that there exists $H$-polar decompositions with $AX = XA$ for matrices which do not admit any $H$-polar decomposition with commuting factors. 

\begin{ex}\label{ex:semi_tot}\normalfont
Consider
\[
X = \begin{bmatrix}
0 & 0 & 1 & 0 & 0\\
0 & 0 & 0 & 1 & 1\\
0 & 0 & 0 & 1 & 0\\
0 & 0 & 0 & 0 & 0\\
0 & 0 & 0 & 0 & 0
\end{bmatrix}\quad\text{ and }\quad H = \begin{bmatrix}
0 & 0 & 0 & 1 & 0\\
0 & 0 & 0 & 0 & 1\\
0 & 0 & 1 & 0 & 0\\
1 & 0 & 0 & 0 & 0\\
0 & 1 & 0 & 0 & 0
\end{bmatrix}.
\]
Then $X$ is $H$-normal and $\ker X = \ker X^{[*]}$ and thus there exists an $H$-polar decomposition with commuting factors. However, $X = UA$ with
\[
U = \begin{bmatrix}
1 & -\frac{1}{2} & \frac{1}{4} & -\frac{1}{32} & 0\\[7pt]
0 & 1 & \frac{1}{2} & -\frac{3}{16} & -\frac{1}{8}\\[7pt]
0 & 0 & 1 & -\frac{1}{2} & -\frac{1}{2}\\[7pt]
0 & 0 & 0 & 1 & 0\\[7pt]
0 & 0 & 0 & \frac{1}{2} & 1
\end{bmatrix},\quad A = \begin{bmatrix}
0 & 0 & 1 & 0 & \frac{1}{2}\\[2pt]
0 & 0 & 0 & \frac{1}{2} & 1\\[2pt]
0 & 0 & 0 & 1 & 0\\[2pt]
0 & 0 & 0 & 0 & 0\\[2pt]
0 & 0 & 0 & 0 & 0
\end{bmatrix}
\]
is an $H$-polar decomposition with $AX = XA$ and $XU \neq UX$.
\end{ex}

\begin{ex}\normalfont
\begin{comment}
Consider
\[
X = \begin{bmatrix}
0 & 0 & 0 & 1 & 0\\
0 & 0 & 1 & 0 & 0\\
0 & 0 & 0 & 1 & 0\\
0 & 0 & 0 & 0 & 0\\
0 & 0 & 0 & 0 & 0
\end{bmatrix}\quad\text{ and }\quad H = \begin{bmatrix}
0 & 0 & 0 & 0 & 1\\
0 & 0 & 0 & 1 & 0\\
0 & 0 & 1 & 0 & 0\\
0 & 1 & 0 & 0 & 0\\
1 & 0 & 0 & 0 & 0
\end{bmatrix}.
\]
Then $\ker X \neq \ker X^{[*]}$ and thus there exists no $H$-polar decomposition with commuting factors. However, $X = UA$ with 
\[
U = \begin{bmatrix}
1 & 0 & 1 & 0 & -\frac{1}{2}\\
0 & 1 & 0 & 0 & 0\\
0 & 0 & 1 & 0 & -1\\
0 & 0 & 0 & 1 & 0\\
0 & 0 & 0 & 0 & 1
\end{bmatrix},\quad A = \begin{bmatrix}
0 & 0 & 0 & 0 & 0\\
0 & 0 & 1 & 0 & 0\\
0 & 0 & 0 & 1 & 0\\
0 & 0 & 0 & 0 & 0\\
0 & 0 & 0 & 0 & 0
\end{bmatrix}
\]
and $AX = XA$.
\end{comment}
Consider the setting of \Href{5.III} from the appendix:
\[
H = \begin{bmatrix}
0 & 0 & 1 & 0 & 0\\
0 & 0 & 0 & 1 & 0\\
0 & 0 & 0 & z & 0\\
0 & 0 & 0 & 0 & 0\\
0 & 0 & 0 & 0 & 0
\end{bmatrix}\quad\text{ and }\quad H = \begin{bmatrix}
0 & 0 & 0 & 1 & 0 \\
0 & 0 & 0 & 0 & 1\\
0 & 0 & 1 & 0 & 0\\
1 & 0 & 0 & 0 & 0\\
0 & 1 & 0 & 0 & 0
\end{bmatrix}
\]
for some $z\in\mathbb{C}$ such that $\vert z\vert = 1$. Then $X$ is $H$-normal with $\ker X \neq \ker X^{[*]}$. Therefore, there cannot exist any $H$-polar decomposition with commuting factors. However, $X = UA$ with
\[
U = \sqrt{z}\begin{bmatrix}
1 & 0 & 0 & 0 & 0\\
0 & -\frac{1}{2}\conj{z} & \conj{z} & 0 & \conj{z}\\
0 & -1 & 1 & 0 & 0\\
0 & 0 & 0 & 1 & 0\\
0 & \conj{z} & 0 & 0 & 0
\end{bmatrix},\quad A = \sqrt{z}\begin{bmatrix}
0 & 0 & \conj{z} & 0 & 0\\
0 & 0 & 0 & 0 &  0\\
0 & 0 & 0 & 1 & 0\\
0 & 0 & 0 & 0 & 0\\
0 & 0 & 0 & 0 & 0
\end{bmatrix}
\]
is an $H$-polar decomposition with $AX = XA$.
\end{ex}

It seems that the information of the commutativity is inscribed in $U$. On the other hand, it has been shown in \cite{mrr} that there exists no $H$-unitary $\widetilde{U}$ in the setting of example \ref{ex:semi_tot} such that $X = \widetilde{U}A$ is an $H$-polar decomposition with commuting factors. In the following, we will give an explanation of this and construct families of $H$-polar decompositions with commuting factors from a given one. The proof will be based on the following result on Witt extensions from \cite[Section 2]{bmrrr1}.

\begin{theo}[Extended Witt's theorem]\label{theo:witt_extend}
Consider the setting of Theorem \ref{theo:witt}. Let then
\[\mathcal{E} = \left( e_1,\dots, e_{m_0}, e_{m_0+1}, \dots, e_{m}, \widetilde{e}_1, \dots, \widetilde{e}_{m_0}, e_{2m_0+m_++m_-+1}, \dots, e_n\right),
\]
$\mathcal{F} = \left( Ue\right)_{e\in\mathcal{E}}$ and the corresponding matrices

\[E = [e_1,\dots,e_m, \widetilde{e}_1,\dots, \widetilde{e}_{m_0}, e_{2m_0+m_++m_-+1},\dots,e_n],
\]
$F = UE\in\mathbb{F}^{n,n}$, where $m = m_++m_-+m_0$, be the bases from the vectors constructed in the proof of Theorem \ref{theo:witt}. Furthermore, let $J_1\in\mathbb{F}^{m_++m_-,m_++m_-}$ and $J_2\in\mathbb{F}^{n-m-m_0,n-m-m_0}$ be such that
\[
G := \begin{bmatrix}
0 & 0 & I_{m_0} & 0\\
0 & J_1 & 0 & 0\\
I_{m_0} & 0 & 0 & 0\\
0 & 0 & 0 & J_2
\end{bmatrix} = E^*HE = F^*H{F}
\]
is the Gramian matrix of $\mathcal{E}$ and $\mathcal{F}$ with respect to $[\cdot,\cdot]$. Then $J_1$ is diagonal such that its first $m_+$ diagonal elements are $+1$ and its remaining $m_-$ diagonal elements are $-1$. Similarly, we can and do assume that $J_2$ is a diagonal matrix for which some diagonal entries are $+1$ and the remaining are $-1$. In the following, $X_{B_1,B_2}$ denotes the representation of the matrix $X\in\mathbb{F}^{n,n}$ in the bases $B_1$ and $B_2$.

Then $\widetilde{U}$ is another $H$-unitary extension of $U_0$ \iff
\begin{equation}\label{eq:witt_ext}
\widetilde{U}_{\mathcal{E},\mathcal{F}} = \begin{bmatrix}
I_{m_0} & 0 & -\frac{1}{2}P_2^*J_2P_2 + P_3 & -P_2^*J_2P_1\\
0 & I_{m-m_0} & 0 & 0\\
0 & 0 & I_{m_0} & 0\\
0 & 0 & P_2 & P_1
\end{bmatrix},
\end{equation}
where $P_1\in\mathbb{F}^{n-m-m_0,n-m-m_0}$ is $J_2$-unitary, $P_2\in\mathbb{F}^{n-m-m_0,m_0}$ and where $P_3\in\mathbb{F}^{m_0,m_0}$ is skew-adjoint.
\end{theo}
\begin{proof}
Since by definition $\widetilde{U}e_i = f_i$ for $i=1,\dots,m$, we already know that
\[
\widetilde{U}_{\mathcal{E},\mathcal{F}} = \begin{bmatrix}
I_{m_0} & 0 & U_1& U_2\\
0 & I_{m-m_0} & U_3 & U_4\\
0 & 0 & U_5 & U_6\\
0 & 0 & U_7 & U_8
\end{bmatrix}
\]
for some matrices $U_i$, $i=1,\dots,6$ with sizes that correspond to the block partition in \eqref{eq:witt_ext}. The remaining equalities follow from the necessary and sufficient condition
\[
I_n = H^{-1}\widetilde{U}^*H\widetilde{U} = {E}G^{-1}\widetilde{U}_{\mathcal{E},\mathcal{F}}^*G\widetilde{U}_{\mathcal{E},\mathcal{F}}{E}^{-1}
\]
for $\widetilde{U}$ to be $H$-unitary. First, it gives
\[
I_n = G^{-1}\widetilde{U}_{\mathcal{E},\mathcal{F}}^*G\widetilde{U}_{\mathcal{E},\mathcal{F}} = G\widetilde{U}_{\mathcal{E},\mathcal{F}}^*G\widetilde{U}_{\mathcal{E},\mathcal{F}}
\]
and hence
\[
\begin{bmatrix}
U_5^* & U_3^*J_1 & u_{13} & u_{14}\\
0 & I_{m_++m_-} & U_3 & U_4\\
0 & 0 & U_5 & U_6\\
J_2U_6^* & J_2U_4^*J_1 & u_{43} & u_{44}
\end{bmatrix} = \begin{bmatrix}
I_{m_0} & 0 & 0 & 0\\
0 & I_{m_++m_-} & 0 & 0\\
0 & 0 & I_{m_0} & 0\\
0 & 0 & 0 & I_{n-m-m_0}
\end{bmatrix},
\]
where
\begin{align*}
u_{13} &= U_5^*U_1 + U_3^*J_1U_3 + U_1^*U_5 + U_7^*J_2U_7,\\
u_{14} &= U_5^*U_2 + U_3^*J_1U_4 + U_1^*U_6 + U_7^*J_2U_8,\\
u_{43} &= J_2U_6^*U_1 + J_2U_4^*J_1U_3 + J_2U_2^*U_5 + J_2U_8^*J_2U_7\\
u_{44} &= J_2U_6^*U_2 + J_2U_4^*J_1U_4 + J_2U_2^*U_6 + J_2U_8^*J_2U_8.
\end{align*}
The theorem is then derived by equating the blocks.
\end{proof}

\begin{prop}
Let $X\in\mathbb{F}^{n,n}$ and let $X = UA$ be an $H$-polar decomposition with commuting factors. If $X = \widetilde{U}A$ is another $H$-polar decomposition of $X$, then $\widetilde{U}A = A\widetilde{U}$.
\end{prop}
\begin{proof}
Assume without loss of generality that $U$ is the extension from the proof of Theorem \ref{theo:witt} and let $\mathcal{E}$ and $\mathcal{F}$ be as before. (If $U$ is another Witt extension, then it can be constructed analogously as in the proof with corresponding bases $\mathcal{E}$ and $\mathcal{F}$.) By definition, $(e_1,\dots,e_m)$ spans $\ran A$ and
\[
H^{-1}A^*H = A\quad\text{ or }\quad A_{\mathcal{E},\mathcal{E}}^*G = GA_{\mathcal{E},\mathcal{E}}.
\] 
This gives
\[ 
A_{\mathcal{E},\mathcal{E}} = \begin{bmatrix}
0 & A_1 & A_2 & 0\\
0 & A_3 & A_4 & 0\\
0 & 0 & 0 & 0\\
0 & 0 & 0 & 0
\end{bmatrix}
\]
for some matrices $A_i$, $i=1,\dots,4$, with sizes that correspond to the block partition in \eqref{eq:witt_ext}. Moreover, $U_{\mathcal{E},\mathcal{F}} = I_n$ and hence $\widetilde{U} = UM$ for some matrix $M$ with $M_{\mathcal{E},\mathcal{E}}$ in the form of \eqref{eq:witt_ext}. This gives
\[
M_{\mathcal{E},\mathcal{E}}A_{\mathcal{E},\mathcal{E}} = A_{\mathcal{E},\mathcal{E}} = A_{\mathcal{E},\mathcal{E}}M_{\mathcal{E},\mathcal{E}}.
\]
We conclude that
\begin{align*}
\left[\widetilde{U}A\right]_{\mathcal{E},\mathcal{F}} &= U_{\mathcal{E},\mathcal{F}} M_{\mathcal{E},\mathcal{E}}A_{\mathcal{E}, \mathcal{E}}\\
&= U_{\mathcal{E},\mathcal{F}} A_{\mathcal{E},\mathcal{E}} M_{\mathcal{E},\mathcal{E}}\\
&= [UA]_{\mathcal{E},\mathcal{F}} M_{\mathcal{E},\mathcal{E}} \\
&= [AU]_{\mathcal{E},\mathcal{F}}M_{\mathcal{E},\mathcal{E}} \\
&= A_{\mathcal{F}, \mathcal{F}} U_{\mathcal{E},\mathcal{F}} M_{\mathcal{E},\mathcal{E}} = \left[A\widetilde{U}\right]_{\mathcal{E},\mathcal{F}}
\end{align*}
and thus $\widetilde{U}A = A\widetilde{U}$.
\end{proof}

\subsection{Indecomposable matrices and number of polar decompositions with commuting factors}\label{ssec:indecomp}\label{ssec:indecomposables}

Since our definition of $H$-polar decompositions is not restrictive, we do not obtain any uniqueness. However, the additional restriction of commuting factors reduces the number of $H$-polar decomposition drastically. It is shown in this thesis that the selfadjoint factor in the decomposition of Example \ref{ex:semi_tot} is unique up to a sign (cf. Appendix \ref{sec*:Appendix}).

In the following, we will discuss possible ways to determine the number of $H$-polar decompositions with commuting factors.\\

Since we do not want to take into account the number of Witt extensions of every feasible matrix $A$ (this has already be done in \cite[Section 2]{bmrrr1}), we will use the previous result to classify $H$-polar decompositions.
\begin{defin}
Let $X\in\mathbb{F}^{n,n}$. We denote by $CPD(X,H) := CPD_\mathbb{F}(X,H)$ the cardinal of matrices $A\in\mathbb{F}^{n,n}$ such that there exists an $H$-polar decomposition $ X= UA$ with commuting factors.
\end{defin}

\begin{rem}\normalfont
The number $CPD(X,H)$ is the number of equivalence classes with respect to
\[
(U,A) \sim (V, B)\quad:\Leftrightarrow\quad A = B
\]
on the set $\llbrace (U,A)\;\vert\; X = UA \text{ is $H$-polar decomposition with commuting factors}\rrbrace$.
\end{rem}

\begin{defin}[Unitary equivalence]
Let $N_1,N_2,H_1,H_2\in\mathbb{F}^{n,n}$ such that $H_1$ and $H_2$ are Hermitian nonsingular. If there exists a nonsingular $T\in\mathbb{F}^{n,n}$ such that
\[
N_2 = T^{-1}N_1T\quad\text{ and }\quad H_2 = T^*H_1T,
\]
then $(N_1,H_1)$ and $(N_2,H_2)$ are called \emph{unitarily equivalent}.
\end{defin}

Unitary equivalence is the natural equivalence relation on indefinite inner product spaces: It is immediate that $N_1$ is $H_1$-selfadjoint, $H_1$-unitary or $H_1$-normal if and only if $N_2$ is $H_2$-selfadjoint, $H_2$-unitary or $H_2$-normal.

\begin{prop}
Let $X_1,X_2,H_1,H_2\in\mathbb{F}^{n,n}$ be such that $H_1$ and $H_2$ are Hermitian nonsingular. Furthermore, let $T\in\mathbb{F}^{n,n}$ be nonsingular such that $X_2 = T^{-1}X_1T$ and $H_2 = T^*H_1T$. Then $X_1 = U_1A_1$ is an $H_1$-polar decomposition (with commuting factors) of $X_1$ \iff $X_2 = U_2A_2$ with $U_2 = T^{-1}U_1T$, $A_2 = T^{-1}A_1T$ is an $H_2$-polar decomposition (with commuting factors).
\end{prop}
\begin{proof}
Let $X_1 = U_1A_1$ be an $H_1$-polar decomposition of $X_1$. Then
\[
X_2 = T^{-1}X_1T = T^{-1}U_1A_1T = \left(T^{-1}U_1T\right)\cdot \left(T^{-1}A_1T\right) = U_2A_2.
\]
Furthermore, it follows from the above observation that $U_2$ is $H_2$-unitary and $A_2$ is $H_2$-selfadjoint. We conclude that $X_2 = U_2A_2$ is an $H_2$-polar decomposition. For the other implication, it suffices to exchange $T$ with $T^{-1}$.

Clearly $A_1U_1 = U_1A_1$ \iff $A_2U_2 = U_2A_2$.
\end{proof}
\begin{corol}
Let $X_1,X_2,H_1,H_2\in\mathbb{F}^{n,n}$ such that $H_1$ and $H_2$ are Hermitian nonsingular. If $(X_1,H_1)$ and $(X_2,H_2)$ are unitarily equivalent, then
\[
CPD(X_1,H_1) = CPD(X_2,H_2).
\]
\end{corol}

Hence, it suffices to restrict ourselves to one represent of each equivalence class with respect to unitary equivalence. With Theorem \ref{theo:commuting_factors}, we can restrict our attention further to $H$-normal matrices only. 

\begin{defin}
A matrix $X\in\mathbb{F}^{n,n}$ is called \emph{$H$-decomposable}, if there exists a nonsingular $T\in\mathbb{F}^{n,n}$ such that
\[
T^{-1}XT = X_1\oplus X_2\quad \text{ and }\quad T^*HT = H_1\oplus H_2
\]
for some $X_1,H_1\in\mathbb{F}^{m,m}$ and $X_2,H_2\in\mathbb{F}^{n-m,n-m}$, $0 < m < n$. Otherwise, $X$ is called \emph{$H$-indecomposable}.
\end{defin}

Clearly, any matrix $X\in\mathbb{F}^{n,n}$ can always be decomposed as
\begin{equation}\label{eq:decomp_indecomp}
T^{-1}XT = X_1\oplus\dots\oplus X_p,\quad T^*HT = H_1\oplus\dots\oplus H_p
\end{equation}
where $X_j$ is $H_j$-indecomposable, $j=1,\dots,p$.

\begin{lem}\label{lem:standard_decomp}
Let $X\in\mathbb{F}^{n,n}$ and $H$ be as in \eqref{eq:decomp_indecomp}. If $X_j = U_jA_j$ are $H_j$-polar decompositions (with commuting factors), $j=1,\dots,p$, then
\[
T^ {-1}XT = (\underbrace{U_1\oplus\dots\oplus U_p}_{=:U})\cdot (\underbrace{A_1\oplus\dots\oplus A_p}_{=:A})
\]
is a $T^ *HT$-polar decomposition (with commuting factors).
\end{lem}
\begin{proof}
Let $\widetilde{H} := T^{*}HT$. First,
\[
U^*\widetilde{H}U = (U_1^*H_1U_1)\oplus\dots\oplus(U_p^*H_pU_p) = H_1\oplus\dots\oplus H_p = \widetilde{H}
\]
and
\[
\widetilde{H}^{-1}A^*\widetilde{H} = (H_1^{-1}A_1^*H_1)\oplus\dots\oplus (H_p^{-1}A_p^*H_p) = A_1\oplus\dots\oplus A_p = A.
\]
Thus, $U$ is $\widetilde{H}$-unitary and $A$ is $\widetilde{H}$-selfadjoint. Finally
\[
UA = (U_1A_1)\oplus\dots\oplus(U_pA_p) = X_1\oplus\dots\oplus X_p = X.
\]
If $U_jA_j = A_jU_j$, $j=1,\dots,p$, then clearly $UA = AU$.
\end{proof}
\begin{lem}
Let $X\in\mathbb{F}^{n,n}$ and $H$ be as in \eqref{eq:decomp_indecomp}. Then 
\begin{enumerate}
	\item $X$ is $H$-normal \iff $X_j$ is $H_j$-normal, $j = 1,\dots, p$
	\item and $\ker X = \ker X^{[*]_H}$ \iff $\ker X_j = \ker X_j^{[*]_{H_j}}$, $j=1,\dots,p$.
\end{enumerate}
\end{lem}
\begin{proof}
\begin{enumerate}
	\item Note that
\[
X^{[*]_H}X = T\cdot \Big{(}(X_1^{[*]_{H_1}}X_1)\oplus\dots\oplus(X_p^{[*]_{H_p}}X_p)\Big{)}\cdot T^{-1}
\]
and
\[
XX^{[*]_H} =T\cdot \Big{(}(X_1X_1^{[*]_{H_1}})\oplus\dots\oplus(X_pX_p^{[*]_{H_p}})\Big{)}\cdot T^{-1}.
\]
The assertion is derived by equating the blocks.
	\item Let $v = [v_1,\dots, v_p]^T\in\mathbb{F}^n$ such that $X_jv_j$ is defined. Then
	\[
	T^{-1}XTv = [X_1v_1, \dots, X_pv_p]^T\quad\text{ and }\quad T^{-1}X^{[*]_{H}}Tv = [X_1^{[*]_{H_1}}v_1,\dots, X_p^{[*]_{H_p}}v_p]^T.
	\]
	Clearly $Xv = X^{[*]_H}v$ \iff $X_jv_j = X_j^{[*]_{H_j}}v_j$ for $j=1,\dots,p$.\qedhere
\end{enumerate}
\end{proof}
\begin{prop}
Let $\mathbb{F} = \mathbb{C}$ and  let $X,H\in\mathbb{F}^{n,n}$ be as in \eqref{eq:decomp_indecomp}. If $X$ is $H$-normal and $\ker X = \ker X^{[*]_H}$, then
\[
CPD(X, H) \geq 2^p.
\]
\end{prop}
\begin{proof}
If $X_j = U_jA_j$ is an $H_j$-polar decomposition with commuting factors, then 
\begin{equation}\label{eq:oppos_decomp}
X_j = (-U_j)(-A_j)
\end{equation}
is also an $H_j$-polar decomposition with commuting factors. Furthermore, $X_j$ is $H_j$ normal and $\ker X_j = \ker X_j^{[*]_{H_j}}$, $j=1,\dots,p$. Thus, there exists (at least) one $H_j$-polar decomposition with commuting factors (cf. Theorem \ref{theo:commuting_factors}) and its opposite decomposition \eqref{eq:oppos_decomp}. The assertion follows directly.
\end{proof}

\section{Concluding Remarks and Open Problems}

In this last section we will discuss open problems concerning polar decompositions in indefinite inner product space with special interest in polar decompositions with commuting factors. Many of these problems were already known, some became apparent within this thesis.

The examples in the Appendix have answered some questions and raised others. But since not all factorizations are known, the completion of the list might answer some of the open problems or lead to new questions.\\

For the first problem, recall that a matrix $X\in\mathbb{F}^{n,n}$ admits an $H$-polar decomposition \iff there exists an $H$-selfadjoint matrix $A\in\mathbb{F}^{n,n}$ such that $A^2 = X^{[*]}X$ and ${\ker A = \ker X}$. We have already seen, that there does not always exist an $H$-polar decomposition, since $X^{[*]}X$ might not admit any square root. However, an example proving that an $H$-polar decomposition does not exist, when no \emph{$H$-selfadjoint} square root exists, was not yet found. The same question arises exchanging the conditions $\ker A = \ker X$ and $A^{[*]} = A$.

\begin{prob}[Existence of $H$-selfadjoint square roots]\normalfont
Let $X,Y\in\mathbb{F}^{n,n}$ such that ${Y^2 = X^{[*]}X}$ such that $\ker X = \ker Y$. Does there exist an $H$-selfadjoint matrix $A\in\mathbb{F}^{n,n}$ such that $A^2 = X^{[*]}X$ and $\ker X = \ker A$.
\end{prob}

\begin{prob}[Existence of square roots with same kernel]\normalfont
Let $X\in\mathbb{F}^{n,n}$ and let $Y\in\mathbb{F}^{n,n}$ be $H$-selfadjoint such that ${Y^2 = X^{[*]}X}$. Doest there exist an $H$-selfadjoint matrix $A\in\mathbb{F}^{n,n}$ such that $\ker A = \ker X$ and $A^2 = X^{[*]}X$.
\end{prob}

We will now concentrate on $H$-polar decomposition with commuting factors. In this thesis, we could extend Theorem \ref{theo:commuting_factors} from \cite{mrr} to $\mathbb{F} = \mathbb{R}$ for a specific class of matrices. This raises the following problem:

\begin{prob}\normalfont
Can Theorem \ref{theo:commuting_factors} be extended to a more general class of matrices $X\in\mathbb{R}^{n,n}$ in the case $\mathbb{F} = \mathbb{R}$? What conditions have to be added to the theorem? In particular: Can the theorem be extended to all nonsingular matrices?
\end{prob}

In this thesis, it was shown that the commutativity property of the factors is independent of the chosen $H$-unitary factor. Even if this is only of interest if $X$ is singular, the following problem has to be formulated.
\begin{prob}[Characterisation of $H$-polar decompositions with commuting factors]\normalfont
Is it possible to characterise for a given $X\in\mathbb{F}^{n,n}$ all $H$-selfadjoint factors of $H$-polar decompositions of $X$ with commuting factors? If so, is it possible to use this characterisation for the actual computation of $H$-polar decomposition with commuting factors?
\end{prob}

Closely related to this problem is the following one. We showed, that every $H$-polar decomposition $X = UA$ with commuting factors satisfies $AX = XA$. The reverse implication does not hold in general if $X$ is singular. 
\begin{prob}[Characterisation of $H$-polar compositions with $AX = XA$]\normalfont
Is it possible to characterise for a given $X\in\mathbb{F}^{n,n}$ all $H$-selfadjoint $A\in\mathbb{F}^{n,n}$ factors in an $H$-polar decomposition satisfying $AX = XA$? Assume that $AX = XA$. Is it possible to give a necessary or sufficient condition on $A$ such that $A$ is a factor in an $H$-polar decomposition with commuting factors?
\end{prob}

The examples in the Appendix exhibit some common structure. For this, consider the case $\mathbb{F} = \mathbb{C}$. It seems that the $H$-selfadjoint factors with $AX = XA$ form a connected component and that the $H$-selfadjoint factors of $H$-polar decompositions with commuting factors are unique up to a sign.

\begin{prob}\normalfont
Let $X\in\mathbb{C}^{n,n}$. Do the $H$-selfadjoint factors of $H$-polar decompositions with $AX = XA$ form a (path-)connected component? Is $CPD(X,H) = 2$ if $X$ is $H$-indecomposable?
\end{prob}

More general, one can search for an upper bound for the cardinal of $H$-polar decompositions with commuting factors and improve the lower bound that we proved in this thesis. This leads to multiple problems related to the structure of $H$-polar decompositions of $H$-decomposable matrices.

\begin{prob}[$H$-polar decompositions of $H$-decomposable matrices]\normalfont
Let $X\in\mathbb{F}^{n,n}$ with $p$ indecomposable factors. Is the lower bound $CPD(X,H) \geq 2^p$ sharp, i.e. does there exist for every $p\in\mathbb{N}\setminus\llbrace 0\rrbrace$ a matrix $X$ such that $CPD(X, H) = 2^p$? Does there exist an upper bound? Is $CPD(X,H) = 2^p$, i.e. can every $H$-polar decomposition of $X$ be decomposed in $p$ indecomposable $H$-polar decompositions of the indecomposables of $X$?
\end{prob}\hfill\\

What can we get out of this thesis? To answer this, we need to assume that the above problems can be answered as we would like it. Furthermore, we can only speak about $H$-normal matrices $X\in\mathbb{C}^{n,n}$.

We have seen that the generalised polar decomposition has commuting factors, whenever it exists. And investigating the algorithm from Theorem \ref{theo:approx}, we even obtain a \textquotedblleft standard\textquotedblright{} $H$-polar decomposition in the sense of Lemma \ref{lem:standard_decomp}, i.e. an $H$-polar decomposition which preserves the structure of the original matrix. The problem with the generalised polar decomposition is that it does not exist for matrices $X$ such that $X^{[*]}X$ has nonpositive eigenvalues. In the case that $X^{[*]}X$ has no negative eigenvalues, we can still use the canonical generalised polar decomposition with the small disadvantage that the $H$-unitary factor is not uniquely determined. But as soon as $X^{[*]}X$ has one negative eigenvalue, no (canonical) generalised polar decomposition does exist. To define the $H$-polar decomposition for those matrices in a unique way, we thus need another constraint than the restriction on the spectrum of the $H$-selfadjoint factor.

Assume now that $H$-indecomposable matrices indeed have exactly two $H$-polar decompositions with commuting factors. Since we cannot arguably choose between these two decompositions, we will consider them equivalent for now. Nevertheless, we will say for simplicity in this case that the $H$-polar decomposition is uniquely determined. By requesting an $H$-polar decomposition of an $H$-decomposable to have the structure preserving form from Lemma \ref{lem:standard_decomp}, we get exactly $2^p$ possible candidates for $H$-decomposable matrices $X = X_1\oplus\cdots\oplus X_p$. To choose between these $2^p$ decompositions, it would suffices to choose between the two $H$-polar decompositions with commuting factors of $H$-indecomposable matrices. For now, we assume that this choice can (somehow) be done. In this case, we would have defined a unique $H$-polar decomposition on the set of $H$-normal matrices with $\ker X = \ker X^{[*]}$! And furthermore, we have shown that it coincides with the (canonical) generalised polar decomposition whenever it exists. In other words, we could extend the definition of a unique $H$-polar decomposition to a larger class of matrices.

There are obviously still many problems to solve. On the one hand, we used many assumptions for our perspective, and on the other hand, there is no known condition to choose between the two $H$-polar decompositions of $H$-indecomposable matrices. Another problem becomes apparent, when we try to compute such a decomposition. Whereas the sequence of Theorem \ref{theo:approx} can be extended to singular matrices, it does not work for matrices with $X^{[*]}X$ having negative eigenvalues. (In fact, the sequence is not even necessarily well-defined in that case.) In other words: there is no possibility to compute an $H$-polar decomposition with commuting factors in the case that $X^{[*]}X$ has negative eigenvalues. 

We can conclude that a lot has to be done to realise the perspective of $H$-polar decompositions with commuting factors as constraints for a unique $H$-polar decomposition. Apart from the above problems, it seems promising to investigate $H$-polar decompositions with commuting factors for structured matrices, i.e. to examine if $H$-unitary, $H$-skewadjoint or otherwise structured matrices admit $H$-polar decompositions with commuting factors and if they can easily be computed.

\cleardoublepage
\bibliography{lit}
\cleardoublepage

\begin{appendices}
\section{List of $H$-polar decompositions with commuting factors of singular indecomposable matrices}

\subsection{Preliminaries}

Recall from section \ref{ssec:indecomp} that one can decompose every matrix as
\[
T^ {-1}XT = X_1\oplus\dots\oplus X_p,\quad T^ *HT = H_1\oplus\dots\oplus H_p,
\]
where $X_j$ is $H_j$-indecomposable, $j=1,\dots,p$. Thus, the question of classifying the indecomposable matrices arises naturally. Generally, it is not possible to determine all normal forms for the unitary equivalence, as this has been shown to be a wild problem. However, O. Holtz and V. Strauss were able to find all normal forms for $H$-normal matrices in the case where $H$ only has one or two negative eigenvalues (cf. \cite{hs} and \cite{HS}).

In \cite{lmmr}, $H$-polar decompositions are given for these indecomposable $H$-normal matrices. The purpose of this appendix is to examine the relation between $H$-polar decompositions with commuting factors and those with the weaker property $AX = XA$. For this, we will list $H$-polar decompositions with commuting factors and $H$-polar decompositions satisfying $AX = XA$ for indecomposable $H$-normal matrices. Since both properties are equivalent if $X$ is invertible, we will concentrate on the singular indecomposable $H$-normal matrices.

For readability, we denote by \emph{commuting $H$-polar decompositions} $H$-polar decompositions with commuting factors and by \emph{semicommuting $H$-polar decompositions} $H$-polar decompositions with the weaker property $AX = XA$.\\

We consider four cases (in Sections \ref{sec:real_1} through \ref{sec:complex_2}, respectively): $\mathbb{F} = \mathbb{R}$ and $H$ has exactly one negative eigenvalue; $\mathbb{F} = \mathbb{R}$ and $H$ has exactly two negative eigenvalues; $\mathbb{F} = \mathbb{C}$ and $H$ has exactly one eigenvalue; $\mathbb{F} = \mathbb{C}$ and $H$ has exactly two negative eigenvalues. With the exception of \Href{5.II}, we always provide a complete list of semicommuting $H$-polar decompositions. As the proof of the completeness of the lists would reduce the readability, they are found in Appendix \ref{sec*:Appendix} at the end of the thesis. \\

We will use the following notation: $Z_p=\left[\delta_{i+j,p+1}\right]_{i,j=1,\dots,p}$ is the $p\times p$ matrix with ones on the southwest-northeast diagonal and zeros elsewhere; $I_p$ is the $p\times p$ identity matrix. $\Re (z) = \dfrac{z + \conj{z}}{2}$ and $\Im (z) = \dfrac{z - \conj{z}}{2i}$ stand for the real and the imaginary part of a complex number $z\in\mathbb{C}$, respectively.

\subsection{The real case; one negative eigenvalue}\label{sec:real_1}

A complete classification of indecomposable normals in the real case with $H$ having one negative eigenvalue is given in \cite[Section 3]{hs}. We have the following list of singular types.

\newType
%\subsubsection*{Type 2.I}\label{type:2.I}
\[
X = \begin{bmatrix}
0 & 0\\
0 & \lambda
\end{bmatrix},\quad \lambda > 0, \quad H = Z_2.
\]
No semicommuting $H$-polar decomposition exists.

\newType
%\subsubsection*{Type 2.II}\label{type:2.II}
\[
X = \begin{bmatrix}
0& z\\
0&0
\end{bmatrix},\quad z\in\llbrace -1, 1\rrbrace,\quad H = Z_2.
\]
Every semicommuting $H$-polar decomposition has the form
\[
U = \begin{bmatrix}
\frac{z}{s} & 0\\
0 & sz
\end{bmatrix},\qquad A = \begin{bmatrix}
0 & s\\
0&0
\end{bmatrix},\quad s\in\mathbb{R}\setminus\llbrace0\rrbrace.
\]
It commutes if and only if $s \in\llbrace -1 , 1\rrbrace$.

\newType
%\subsubsection*{Type 2.III}\label{type:2.III}
\[
X = \begin{bmatrix}
0 & 1 & 0\\
0 & 0 & 1\\
0 & 0 & 0
\end{bmatrix},\quad H = Z_3.
\]
Every semicommuting $H$-polar decomposition has the form
\[
U = \begin{bmatrix}
\epsilon & -s & -\frac{\epsilon s^2}{2}\\
0 & \epsilon & s\\
0 & 0 & \epsilon
\end{bmatrix},\qquad A = \begin{bmatrix}
0 & \epsilon & s\\
0 & 0 & \epsilon\\
0 & 0 & 0
\end{bmatrix},\quad \epsilon\in\llbrace -1,1\rrbrace, s \in\mathbb{R}.
\]
It commutes if and only if $s = 0$.

\newType
%\subsubsection*{Type 2.IV}\label{type:2.IV}
\[
X = \begin{bmatrix}
0 & 1 & r\\
0 & 0 & -1\\
0 & 0 & 0
\end{bmatrix},\quad r\in\mathbb{R},\quad H = Z_3.
\]
No semicommuting $H$-polar decomposition exists.

\newType
%\subsubsection*{Type 2.V}\label{type:2.V}
\[
X = \begin{bmatrix}
0 & 1 & 0 & 0\\
0 & 0 & 0 & \cos\alpha\\
0 & 0 & 0 & \sin\alpha\\
0 & 0 & 0 & 0
\end{bmatrix},\quad 0< \alpha<\pi, \quad H = \begin{bmatrix}
0 & 0 & 1\\
0 & I_2 & 0\\
1 & 0 & 0
\end{bmatrix}.
\]
No semicommuting $H$-polar decomposition exists.

\subsection{The real case; two negative eigenvalues}\label{sec:real_2}
A complete classification of indecomposable normals in the real case with $H$ having two negative eigenvalues is given in \cite{hs}. We only consider the following list of singular types.

\newType
%\subsubsection*{Type 3.I}\label{type:3.I}
\[
X = \begin{bmatrix}
0 & 1 & 0 & 0\\
0 & 0 & z & 0\\
0 & 0 & 0 & 1\\
0 & 0 & 0 & 0
\end{bmatrix},\quad z\in\llbrace -1,1\rrbrace,\quad H = Z_4.
\]
Every semicommuting $H$-polar decomposition has the form
\[
U = \begin{bmatrix}
\epsilon & 0 & -s & 0\\
0 & \epsilon & 0 & s\\
0 & 0 & \epsilon & 0\\
0 & 0 & 0 & \epsilon
\end{bmatrix},\qquad A = \begin{bmatrix}
0 & \epsilon & 0 & s\\
0 & 0 & \epsilon z & 0\\
0 & 0 & 0 &\epsilon\\
0 & 0 & 0 & 0
\end{bmatrix},\quad \epsilon\in\llbrace -1,1\rrbrace, s\in\mathbb{R}.
\]
These $H$-polar decompositions are commuting if and only if $s = 0$.

\newType
%\subsubsection*{Type 3.II}\label{type:3.II}
\[
X = \begin{bmatrix}
0 & 1 & 0 & 0 & 0\\
0 & 0 & 1 & 0 & 0\\
0 & 0 & 0 & 1 & 0\\
0 & 0 & 0 & 0 & 1\\
0 & 0 & 0 & 0 & 0
\end{bmatrix},\quad H = Z_5.
\]
Every semicommuting $H$-polar decomposition has the form
\[
U = \begin{bmatrix}
\epsilon & 0 & 0 & -s & 0\\
0 & \epsilon & 0 & 0 & s\\
0 & 0 & \epsilon & 0 & 0\\
0 & 0 & 0 & \epsilon & 0\\
0 & 0 & 0 & 0 & \epsilon
\end{bmatrix}, \qquad A = \begin{bmatrix}
0 & \epsilon & 0 & 0 & s\\
0 & 0 & \epsilon & 0 & 0\\
0 & 0 & 0 & \epsilon & 0\\
0 & 0 & 0 & 0 & \epsilon\\
0 & 0 & 0 & 0 & 0
\end{bmatrix},\quad \epsilon\in\llbrace -1,1\rrbrace, s\in\mathbb{R}.
\]
These $H$-polar decompositions are commuting if and only if $s = 0$.

\newType
%\subsubsection*{Type 3.III}\label{type:3.III}
\[
X = \begin{bmatrix}
0 & 1 & -r & 0 & s\\
0 & 0 & 1 & r & 0\\
0 & 0 & 0 & -1 & -r\\
0 & 0 & 0 & 0 & -1\\
0 & 0 & 0 & 0 & 0
\end{bmatrix},\quad r,s\in\mathbb{R},\quad H = Z_5.
\]
No semicommuting $H$-polar decomposition exists.

\newType
%\subsubsection*{Type 3.IV}\label{type:3.IV}
\[
X = \begin{bmatrix}
0 & 1 & 0 & 0\\
0 & 0 & 0 & z\\
0 & 0 & 0 & 0\\
0 & 0 & 0 & 0
\end{bmatrix},\quad z\in\llbrace -1,1\rrbrace,\quad H = Z_4.
\]
No semicommuting $H$-polar decomposition exists.

\newType
%\subsubsection*{Type 3.V}\label{type:3.V}
\[
X = \begin{bmatrix}
0 & 1 & z & 0\\
0 & 0 & 0 & r\\
0 & 0 & 0 & \frac{z}{r}\\
0 & 0 & 0 & 0
\end{bmatrix},\quad z\in\llbrace -1,1\rrbrace, r\in\mathbb{R}, \vert r\vert > 1,\quad H = Z_4.
\]
No semicommuting $H$-polar decomposition exists.

\newType
%\subsubsection*{Type 3.VI}\label{type:3.VI}
\[
X = \begin{bmatrix}
0 & 1 & 0 & \frac{r^2}{2} & 0\\
0 & 0 & 0 & z & 0\\
0 & 0 & 0 & 0 & r\\
0 & 0 & 0 & 0 & 1\\
0 & 0 & 0 & 0 & 0
\end{bmatrix},\quad z\in\llbrace -1,1\rrbrace, r > 0,\quad H = Z_5.
\]
No semicommuting $H$-polar decomposition exists.

\newType
%\subsubsection*{Type 3.VII}\label{type:3.VII}
\[
X = \begin{bmatrix}
0 & 1 & 0 & 0 & 0 & 0\\
0 & 0 & 1 & 0 & 0 & -\frac{r^2}{2}\\
0 & 0 & 0 & 1 & 0 & 0\\
0 & 0 & 0 & 0 & 0 & 1\\
0 & 0 & 0 & 0 & 0 & r\\
0 & 0 & 0 & 0 & 0 & 0
\end{bmatrix},\quad r > 0,\quad H = \begin{bmatrix}
0 & 0 & 0 & 1\\
0 & Z_3 & 0 & 0\\
0 & 0 & 1 & 0\\
1 & 0 & 0 & 0
\end{bmatrix}.
\]
No semicommuting $H$-polar decomposition exists.

\newType
%\subsubsection*{Type 3.VIII}\label{type:3.VIII}
\[
X = \begin{bmatrix}
0 & 1 & -2r & 0 & 0 & 0\\
0 & 0 & 1 & r & 0 & -2r^2 + \frac{s^2}{2}\\
0 & 0 & 0 & -1 & 0 & 0\\
0 & 0 & 0 & 0 & 0 & -1\\
0 & 0 & 0 & 0 & 0 & s\\
0 & 0 & 0 & 0 & 0 & 0
\end{bmatrix},\quad r,s\in\mathbb{R}, s > 0,\quad H = \begin{bmatrix}
0 & 0 & 0 & 1\\
0 & Z_3 & 0 & 0\\
0 & 0 & 1 & 0\\
1 & 0 & 0 & 0
\end{bmatrix}.
\]
No semicommuting $H$-polar decomposition exists.

\newType
%\subsubsection*{Type 3.IX}\label{type:3.IX}
\[
X = \begin{bmatrix}
0 & 0 & \cos\alpha & \sin\alpha\\
0 & 0 & -\sin\alpha & \cos\alpha\\
0 & 0 & 0 & 0\\
0 & 0 & 0 & 0
\end{bmatrix},\quad 0 < \alpha < \pi,\quad H = \begin{bmatrix}
0 & I_2\\
I_2 & 0
\end{bmatrix}.
\]
Every semicommuting $H$-polar decomposition has the form
\[
U = \begin{bmatrix}
U_{11} & U_{12}(u)\\
0 & U_{22}\\
\end{bmatrix},\quad A = \begin{bmatrix}
0 & 0 & r & s\\
0 & 0 & s & t\\
0 & 0 & 0 & 0\\
\end{bmatrix},\quad r,s,t,u\in\mathbb{R}, rt - s^2\neq 0
\]
with
\[
U_{11} = \dfrac{1}{rt - s^2}\cdot\begin{bmatrix}
t\cos\alpha - s\sin\alpha & r\sin\alpha - s\cos\alpha\\
-t\sin\alpha - s\cos\alpha & s\sin\alpha + r\cos\alpha
\end{bmatrix},
\]
\[
U_{22} = \begin{bmatrix}
r\cos\alpha -s\sin\alpha & s\cos\alpha - t\sin\alpha\\
r\sin\alpha + s\cos\alpha & s\sin\alpha + t\cos\alpha
\end{bmatrix}
\]
and$^*$
\[
U_{12}(u) = \dfrac{u}{s\sin\alpha - r\cos\alpha}\begin{bmatrix}
s\cos\alpha + r\sin\alpha & s\sin\alpha - r\cos\alpha\\
-t\cos\alpha - s\sin\alpha & s\cos\alpha - t\sin\alpha
\end{bmatrix}.
\]
Furthermore, there exists$^*$ no commuting $H$-polar decomposition.

\newType
%\subsubsection*{Type 3.X}\label{type:3.X}
\[
X = \begin{bmatrix}
0 & 0 & 0 & 1\\
0 & 0 & q & 0\\
0 & 0 & 0 & 0\\
0 & 0 & 0 & 0
\end{bmatrix},\quad q\in\mathbb{R}, \vert q\vert > 1,\quad H = \begin{bmatrix}
0 & I_2\\
I_2 & 0
\end{bmatrix}.
\]

Every semicommuting $H$-polar decomposition has the form
\[
U = \begin{bmatrix}
U_{11} & U_{12}(u)\\
0 & U_{22}\\
\end{bmatrix},\quad A = \begin{bmatrix}
0 & 0 & r & s\\
0 & 0 & s & t\\
0 & 0 & 0 & 0\\
\end{bmatrix},\quad r,s,t,u\in\mathbb{R}, rt - s^2\neq 0
\]
with
\[
U_{11} = \dfrac{1}{rt - s^2}\cdot\begin{bmatrix}
-s & r\\
qt & -qs
\end{bmatrix},
\]
\[
U_{22} = \dfrac{1}{q}\cdot\begin{bmatrix}
qs & qt\\
r & s
\end{bmatrix}
\]
and
\[
U_{12}(u) = \begin{bmatrix}
u & 0\\
0 & \frac{-qt}{r}u
\end{bmatrix},\quad \text{ if }s = 0
\]
or
\[
U_{12}(u) = \dfrac{qu}{s}\cdot\begin{bmatrix}
r(tr - s^2) & \frac{s}{q}\\
s(tr - s^2) & -t
\end{bmatrix},\quad \text{ if }s\neq 0.
\]
It commutes if and only if $s = \epsilon\sqrt{q}$, $\epsilon\in\llbrace -1,1\rrbrace$ and $r = t = 0$.

\newType
%\subsubsection*{Type 3.XI}\label{type:3.XI}
\[
X = \begin{bmatrix}
0 & 0 & \frac{z}{2} & z\\
0 & 0 & -z & 0\\
0 & 0 & 0 & 0\\
0 & 0 & 0 & 0
\end{bmatrix},\quad z\in\llbrace -1,1\rrbrace,\quad H = \begin{bmatrix}
0 & I_2\\
I_2 & 0
\end{bmatrix}.
\]
Every semicommuting $H$-polar decomposition has the form
\[
U = \begin{bmatrix}
U_{11} & U_{12}(u,v)\\
0 & U_{22}\\
\end{bmatrix},\quad A = \begin{bmatrix}
0 & 0 & r & s\\
0 & 0 & s & t\\
0 & 0 & 0 & 0\\
\end{bmatrix},\quad r,s,t,u,v\in\mathbb{R}, rt - s^2\neq 0
\]
with
\[
U_{11} = \dfrac{1}{rt - s^2}\cdot\begin{bmatrix}
z\left(\frac{t}{2}-s\right) & z\left(r -\frac{s}{2}\right)\\
-tz & sz
\end{bmatrix},
\]
\[
U_{22} = \begin{bmatrix}
sz & tz\\
z\left(\frac{s}{2} - r\right) & z\left(\frac{t}{2}-s\right)
\end{bmatrix}
\]
and
\[
U_{12}(u,v) = \begin{bmatrix}
0 & 0\\
u & v
\end{bmatrix},\quad \text{ if }t = 2s\text{ and }s = 2r,
\]
\[
U_{12}(u,v) := U_{12}(u) = \begin{bmatrix}
u & 0\\
\frac{2su}{2r - s} & \frac{4su}{s - 2r}
\end{bmatrix},\quad\text{ if }t = 2s\text{ and }s\neq 2r
\]
or
\[
U_{12}(u,v) := U_{12}(u) = \dfrac{u}{2s - t}\cdot\begin{bmatrix}
2r - s & 2s - t\\
2s & 2t
\end{bmatrix} \quad\text{ if }t\neq 2s.
\]
There exist no commuting $H$-polar decomposition.

\subsection{The complex case; one negative eigenvalue}\label{sec:complex_1}

A complete classification of indecomposable normals in the complex case with $H$ having one negative eigenvalue is given in \cite{gr}. We have the following list of singular types.

\newType
%\subsubsection*{Type 4.I}\label{type:4.I}
\[
X = \begin{bmatrix}
\lambda_1 & 0\\
0 & \lambda_2
\end{bmatrix},\quad \lambda_1,\lambda_2\in\mathbb{C}, \lambda_1\neq \lambda_2, \lambda_1\lambda_2 = 0,\quad H = Z_2.
\]
No semicommuting $H$-polar decomposition exists.

\newType
%\subsubsection*{Type 4.II}\label{type:4.II}
\[
X = \begin{bmatrix}
0 & z\\
0 & 0
\end{bmatrix},\quad z\in\mathbb{C}, \vert z \vert = 1,\quad H = Z_2.
\]
Every semicommuting $H$-polar decomposition has the form
\[
U=\begin{bmatrix}
\frac{z}{s} & irsz\\
0 & sz
\end{bmatrix},\qquad A = \begin{bmatrix}
0 & s \\
0 & 0
\end{bmatrix},\quad r,s\in\mathbb{R}, s\neq 0
\]
It commutes if and only if $s \in\llbrace -1,1\rrbrace$.

\newType
%\subsubsection*{Type 4.III}\label{type:4.III}

\[
X = \begin{bmatrix}
0 & z & r\\
0 & 0 & z\\
0 & 0 & 0
\end{bmatrix},\quad r\in\mathbb{R}, z\in\mathbb{C},\; \vert z\vert = 1,\quad H = Z_3.
\]
Every semicommuting $H$-polar decomposition has the form
\[
U = \begin{bmatrix}
\epsilon z & \epsilon r - sz & u\\
0 & \epsilon z & sz - \epsilon z^2 r\\
0 & 0  &\epsilon z
\end{bmatrix},\qquad A = \begin{bmatrix}
0 & \epsilon & s\\
0 & 0 & \epsilon\\
0 & 0 & 0
\end{bmatrix},\quad \epsilon \in\llbrace -1,1\rrbrace, s\in\mathbb{R}
\]
such that $\Re(u\conj{z}) = \dfrac{-\epsilon\cdot \vert s - \epsilon rz\vert^2}{2}$. 

It commutes if and only if $s = \epsilon r\Re(z)$.

\newType
%\subsubsection*{Type 4.IV}\label{type:4.IV}
\[
X = \begin{bmatrix}
0 & 1 & ir\\
0 & 0 & 1\\
0 & 0 & 0
\end{bmatrix},\quad r\in\mathbb{R},\quad H = Z_3.
\]
Every semicommuting $H$-polar decomposition has the form
\[
U = \begin{bmatrix}
\epsilon & \epsilon ir - s & u\\
0 & \epsilon & \epsilon ir + s\\
0 & 0 & \epsilon
\end{bmatrix},\qquad A = \begin{bmatrix}
0 & \epsilon & s\\
0 & 0 & \epsilon\\
0& 0 & 0
\end{bmatrix},\quad \epsilon\in\llbrace -1,1\rrbrace, s\in\mathbb{R}
\]
such that $\Re(u) = \dfrac{-\epsilon(r^2 + s^2)}{2}$.

It commutes if and only if $s = 0$.

\newType
%\subsubsection*{Type 4.V}\label{type:4.V}
\[
X = \begin{bmatrix}
0 & \cos\alpha & \sin\alpha & 0\\
0 & 0 & 0 &1\\
0 & 0 & 0 & 0\\
0 &0 & 0 & 0
\end{bmatrix},\quad 0 < \alpha \leq \dfrac{\pi}{2},\quad H = \begin{bmatrix}
0 & 0 & 0 & 1\\
0 & 1 & 0 & 0\\
0 & 0 & 1 & 0\\
1 & 0 & 0 & 0
\end{bmatrix}.
\]
No semicommuting $H$-polar decomposition exists.

\subsection{The complex case; two negative eigenvalues}\label{sec:complex_2}

A complete classification of indecomposable normals in the complex case with $H$ having two negative eigenvalues is given in \cite[Section 3]{HS}. We only consider the following list of singular types.
\begin{comment}
\newType
%\subsubsection*{Type 5.I}\label{type:5.I}
\[
X = \begin{bmatrix}
\lambda_1 & 1 & 0 & 0\\
0 & \lambda_1 & 0 & 0\\
0 & 0 & \lambda_2 & 0\\
0 & 0 & x & \lambda_2
\end{bmatrix},\quad \lambda_1,\lambda_2,x\in\mathbb{C}, \lambda_1\neq\lambda_2, \lambda_1\lambda_2 = 0,\quad H = \begin{bmatrix}
0 & I_2\\
I_2 & 0
\end{bmatrix}.
\]

\todo{Es fehlt die Zerlegung.}

\newType
%\subsubsection*{Type 5.II}\label{type:5.II}
\[
X = \begin{bmatrix}
0 & 1 & ir & isz\\
0 & 0 & z & 0\\
0 & 0 & 0 & z^2\\
0 & 0 & 0 & 0
\end{bmatrix},\quad r,s\in\mathbb{R}, z\in\mathbb{C}, \vert z\vert = 1,\quad H = Z_4.
\]

\todo{Es fehlt die Zerlegung.}

\newType
%\subsubsection*{Type 5.III}\label{type:5.III}
\[
X = \begin{bmatrix}
0 & 1 & 0 & 0 & it\\
0 & 0 & 1 & ir & -2r^3 + is\\
0 & 0 & 0 & 1 & 2ir\\
0 & 0 & 0 & 0 & 1\\
0 & 0 & 0 & 0 & 0
\end{bmatrix},\quad r,s,t\in\mathbb{R},\quad H = Z_5.
\]

\todo{Es fehlt die Zerlegung.}

\newType
%\subsubsection*{Type 5.IV}\label{type:5.IV}
\[
X = \begin{bmatrix}
0 & 1 & 0 & 0 & it\\
0 & 0 & z & r & -2z^2r^2\Im(z) + isz^2\\
0 & 0 & 0 & z & -2irz^2\Im(z)\\
0 & 0 & 0 & 0 & z^2\\
0 & 0 & 0 & 0 & 0
\end{bmatrix},\quad r,s,t\in\mathbb{R}, z\in\mathbb{C}, \vert z\vert = 1, z\neq i, 0 < \arg z < \pi,\quad H = Z_5.
\]

\todo{Es fehlt die Zerlegung.}

\newType
%\subsubsection*{Type 5.V}\label{type:5.V}
\[
\begin{bmatrix}
0 & 1 & 0 & 0 & t\\
0 & 0 & i & r & 2r^2 + is\\
0 & 0 & 0 & i & 2ir\\
0 & 0 & 0 & 0 & -1\\
0 & 0 & 0 & 0 & 0
\end{bmatrix},\quad r,s,t\in\mathbb{R},\quad H = Z_5
\]
\todo{Es fehlt die Zerlegung.}
\end{comment}
\newType
%\subsubsection*{Type 5.VI}\label{type:5.VI}
\[
\begin{bmatrix}
0 & 1 & 0 & 0\\
0 & 0 & 0 & z\\
0 & 0 & 0 & 0\\
0 & 0 & 0 & 0
\end{bmatrix},\quad z\in\mathbb{C},\vert z\vert = 1,\quad H = Z_4.
\]
No semicommuting $H$-polar decomposition exists.

\begin{comment}
\todo{Es fehlt die Zerlegung.}

\newType
%\subsubsection*{Type 5.VII}\label{type:5.VII}
\[
X = \begin{bmatrix}
0 &  1 & 1 & 0\\
0 & 0 & 0 & z\\
0 & 0 & 0 & (1+ir)z\\
0 & 0 & 0 & 0
\end{bmatrix},\quad r > 0, z\in\mathbb{C}, \vert z\vert = 1,\quad H = Z_4.
\]

\todo{Es fehlt die Zerlegung.}

\newType
%\subsubsection*{Type 5.VIII}\label{type:5.VIII}
\[
X = \begin{bmatrix}
0 & 1 & -1 & 0\\
0 & 0 & 0 & z\\
0 & 0 & 0 & -(1 + ir)z\\
0 & 0 & 0 & 0
\end{bmatrix},\quad r > 0, z\in\mathbb{C}, \vert z\vert = 1,\quad H = Z_4.
\]

\todo{Es fehlt die Zerlegung.}

\newType
%\subsubsection*{Type 5.IX}\label{type:5.IX}
\[
X = \begin{bmatrix}
0 & 1 & 0 & \frac{r^2 + is}{2} & 0\\
0 & 0 & 0 & z & 0\\
0 & 0 & 0 & 0 & r\\
0 & 0 & 0 & 0 & z^2\\
0 & 0 & 0 & 0 & 0
\end{bmatrix},\quad r,s\in\mathbb{R}, r > 0, z\in\mathbb{C}, \vert z\vert = 1,\quad H = Z_5.
\]

\todo{Es fehlt die Zerlegung.}

\newType
%\subsubsection*{Type 5.X}\label{type:5.X}
\[
X = \begin{bmatrix}
0 & 1 & 2ir & 0 & 0 & 0\\
0 & 0 & 1 & ir & 0 & 2r^2  - \frac{s^2}{2} + it\\
0 & 0 & 0 & 1 & 0 & 0\\
0 & 0 & 0 & 0 & 0 & 1\\
0 & 0 & 0 & 0 & 0 & s\\
0 & 0 & 0 & 0 & 0 & 0
\end{bmatrix},\quad r,s,t\in\mathbb{R}, s > 0, \quad H = \begin{bmatrix}
0 & 0 & 0 & 1\\
0 & Z_3 & 0 & 0\\
0 & 0 & 1 & 0\\
1 & 0 & 0 & 0
\end{bmatrix}.
\]

\todo{Es fehlt die Zerlegung.}
\end{comment}

\newType
%\subsubsection*{Type 5.XI}\label{type:5.XI}
\[
X = \begin{bmatrix}
0 & 1 & -2ir\Im(z) & 0 & 0 & 0\\
0 & 0 & z & r & 0 & \left(2r^2\Im(z)^2 - \frac{s^2}{2} +it\right)z^2\\
0 & 0 & 0 & z & 0 & 0\\
0 & 0 & 0 & 0 & 0 & z^2\\
0 & 0 & 0 & 0 & 0 & s\\
0 & 0 & 0 & 0 & 0 & 0
\end{bmatrix},\quad \begin{array}{ll}r,s,t\in\mathbb{R}, \\s > 0, \\z\in\mathbb{C}, \\\vert z\vert = 1, \\0 < \arg z<\pi\end{array},\quad H = \begin{bmatrix}
0 & 0 & 0 & 1\\
0 & Z_3 & 0 & 0\\
0 & 0 & 1 & 0\\
1 & 0 & 0 & 0
\end{bmatrix}.
\]
No semicommuting $H$-polar decomposition exists.

\begin{comment}
\newType
%\subsubsection*{Type 5.XII}\label{type:5.XII}
\[
X = \begin{bmatrix}
0 & 0 & z & \frac{r(1 - i\sqrt{3})z}{2}\\
0 & 0 & 0 & \frac{(1 + i\sqrt{3})z}{2}\\
0 & 0 & 0 & 0\\
0 & 0 & 0 & 0
\end{bmatrix},\quad r \geq \sqrt{3}, z\in\mathbb{C}, \vert z\vert = 1,\quad H = \begin{bmatrix}
0 & I_2\\
I_2 & 0
\end{bmatrix}.
\]

\todo{Es fehlt die Zerlegung.}

\newType
%\subsubsection*{Type 5.XIII}\label{type:5.XIII}
\[
X = \begin{bmatrix}
0 & 0 & 0 & 0\\
0 & 0 & 1 & 0\\
0 & 0 & 0 & 0\\
0 & 0 & 0 & 0
\end{bmatrix},\quad H = \begin{bmatrix}
0 & I_2\\
I_2 & 0
\end{bmatrix}.
\]

\todo{Es fehlt die Zerlegung.}
\end{comment}

\newType
%\subsubsection*{Type 5.XIV}\label{type:5.XIV}
\[
X = \begin{bmatrix}
0 & 0 & 1 & 0 & 0\\
0 & 0 & 0 & 1 & 0\\
0 & 0 & 0 & z & 0\\
0 & 0 & 0 & 0 & 0\\
0 & 0 & 0 & 0 & 0
\end{bmatrix},\quad z\in\mathbb{C}, \vert z\vert = 1,\quad H = \begin{bmatrix}
0 & 0 & I_2\\
0 & 1 & 0\\
I_2 & 0 & 0
\end{bmatrix}.
\]

The matrix may be decomposed by
\[
U = \sqrt{z}\begin{bmatrix}
1 & 0 & 0 & 0 & 0\\
0 & -\frac{1}{2}\conj{z} & \conj{z} & 0 & \conj{z}\\
0 & -1 & 1 & 0 & 0\\
0 & 0 & 0 & 1 & 0\\
0 & \conj{z} & 0 & 0 & 0
\end{bmatrix}\quad\text{ and }\quad A = \sqrt{z}\begin{bmatrix}
0 & 0 & \conj{z} & 0 & 0\\
0 & 0 & 0 & 0 & 0\\
0 & 0 & 0 & 1 & 0\\
0 & 0 & 0 & 0 & 0\\
0 & 0 & 0 & 0 & 0
\end{bmatrix}
\]
into a semicommuting $H$-polar decomposition. (There may be other possible decompositions!)

No commuting $H$-polar decomposition exists.

\begin{comment}
\newType
%\subsubsection*{Type 5.XV}\label{type:5.XV}
\[
X = \begin{bmatrix}
0 & 0 & 1 & 0 & 0\\
0 & 0 & 0 & r & z\\
0 & 0 & 0 & z^2 & 0\\
0 & 0 & 0 & 0 & 0\\
0 & 0 & 0 & 0 & 0
\end{bmatrix},\quad r > 0, z\in\mathbb{C}, \vert z\vert = 1,\quad H = \begin{bmatrix}
0 & 0 & I_2\\
0 & 1 & 0\\
I_2 & 0 & 0
\end{bmatrix}.
\]

\todo{Es fehlt die Zerlegung.}

\newType
%\subsubsection*{Type 5.XVI}\label{type:5.XVI}
\[
X = \begin{bmatrix}
0 & 0 & 1 & 0 & ir & 0\\
0 & 0 & 0 & 1 & s & ir\\
0 & 0 & 0 & 0 & z & 0\\
0 & 0 & 0 & 0 & 0 & z\\
0 & 0 & 0 & 0 & 0 & 0\\
0 & 0 & 0 & 0 & 0 & 0
\end{bmatrix},\quad r,s\in\mathbb{R}, s > 0, z\in\mathbb{C}\setminus\llbrace -1\rrbrace, \vert z\vert = 1,\quad H = \begin{bmatrix}
0 & 0 & I_2\\
0 & I_2 & 0\\
I_2 & 0 & 0
\end{bmatrix}.
\]
\todo{Es fehlt die Zerlegung.}
\end{comment}

\newType
%\subsubsection*{Type 5.XVII}\label{type:5.XVII}
\[
X = \begin{bmatrix}
0 & 0 & 1 & 0 & 0 & 0 & 0\\
0 & 0 & 0 & 1 & 0 & 0 & 0\\
0 & 0 & 0 & 0 & 0 & -z_1\conj{z_2}\cos\alpha & \sin\alpha\cos\beta\\
0 & 0 & 0 & 0 & 0 & z_1\sin\alpha & z_2\cos\alpha\cos\beta\\
0 & 0 & 0 & 0 & 0 & 0 & \sin\beta\\
0 & 0 & 0 & 0 & 0 & 0 & 0\\
0 & 0 & 0 & 0 & 0 & 0 & 0
\end{bmatrix},\quad \begin{array}{ll}z_1,z_2\in\mathbb{C}, \\\vert z_1\vert = \vert z_2\vert = 1\\ 0<\alpha,\beta\leq \dfrac{\pi}{2}\end{array},\quad H = \begin{bmatrix}
0 & 0 & I_2\\
0 & I_3 & 0\\
I_2 & 0 & 0
\end{bmatrix}.
\]
No semicommuting $H$-polar decomposition exists.

\newType
%\subsubsection*{Type 5.XVIII}\label{type:5.XVIII}
\[
X = \begin{bmatrix}
0 & 0 & 1 & 0 & 0 & 0 & 0 & 0\\
0 & 0 & 0 & 1 & 0 & 0 & 0 & 0\\
0 & 0 & 0 & 0 & 0 & 0 & -z_1\conj{z_2}\sin\alpha\cos\beta & \cos\alpha\cos\gamma\\
0 & 0 & 0 & 0 & 0 & 0 & z_1\cos\alpha\cos\beta & z_2\sin\alpha\cos\gamma\\
0 & 0 & 0 & 0 & 0 & 0 & \sin\beta & 0\\
0 & 0 & 0 & 0 & 0 & 0 & 0 & \sin\gamma\\
0 & 0 & 0 & 0 & 0 & 0 & 0 & 0\\
0 & 0 & 0 & 0 & 0 & 0 & 0 & 0
\end{bmatrix},\quad\begin{array}{ll}
z_1,z_2\in\mathbb{C}\\
\vert z_1\vert = \vert z_2\vert = 1\\[6pt]
0\leq \alpha < \dfrac{\pi}{2}\\
0 < \beta < \gamma\leq \dfrac{\pi}{2}
\end{array},\quad H = \begin{bmatrix}
0 & 0 & I_2\\
0 & I_4 & 0\\
I_2 & 0 & 0
\end{bmatrix}.
\]
No semicommuting $H$-polar decomposition exists.

\section{Proofs}\label{sec*:Appendix}

In every case, we assume $X = UA$ to be a semicommuting $H$-polar decomposition. Furthermore, we denote $U = \begin{bmatrix}
u_{i,j}
\end{bmatrix}_{i,j}$ and $A = [a_{i,j}]_{i,j}$. For the proofs, we use the following equations:
\begin{eqnarray}
UA = X\label{eq:UA}\\
A^{[*]} = A\label{eq:selfadj}\\
\ker A = \ker X\label{eq:kernel}\\
A^ 2 = X^ {[*]}X\label{eq:A2}\\
U^*HU = H\label{eq:unitary}\\
X^*HU = U^*HX\label{eq:X_U_H}\\
AX = XA\label{eq_semi}
\end{eqnarray}
Equations \eqref{eq:UA},\eqref{eq:selfadj} and \eqref{eq:unitary} follow from the definition of an $H$-polar decomposition. Equations \eqref{eq:kernel}, \eqref{eq:A2} and \eqref{eq:X_U_H} are necessary conditions as one can easily check. The last equation \eqref{eq_semi} only applies in our setting of semicommuting $H$-polar decompositions.

Note that we write $\cdots$ for terms we do not need to calculate explicitly.\\

Some of the calculations have been done by MAPLE or Sage. These calculations* are asterisked. 
\stepcounter{subsection}
\setcounter{subsection}{2}
\newTypeProof
%\subsubsection*{\Href{2.I}}\label{p:2.I}
Special case of \Pref{4.I}.

\newTypeProof
%\subsubsection*{\Href{2.II}}\label{p:2.II}
Special case of \Pref{4.II}.

\newTypeProof
%\subsubsection*{\Href{2.III}}\label{p:2.III}
Special case of \Pref{4.III}.

\newTypeProof
%\subsubsection*{\Href{2.IV}}\label{p:2.IV}
Recall
\[
X = \begin{bmatrix}
0 & 1 & r\\
0 & 0 & -1\\
0 & 0 & 0
\end{bmatrix},\quad r\in\mathbb{R},\quad H = Z_3.
\]
Equations \eqref{eq:selfadj} and \eqref{eq:kernel} give
\[
A = \begin{bmatrix}
0 & a_{12} & a_{13}\\
0 & a_{22} & a_{12}\\
0 & 0 & 0
\end{bmatrix},\quad a_{12},a_{13},a_{22}\in\mathbb{R}
\]
and thus
\[
AX = \begin{bmatrix}
0 & 0 & -a_{12}\\
0 & 0 & -a_{22}\\
0 & 0 & 0
\end{bmatrix} \quad\text{ and }\quad XA = \begin{bmatrix}
0 & a_{22} & a_{12}\\
0 & 0 & 0\\
0 & 0 & 0
\end{bmatrix}.
\]
With \eqref{eq_semi}, this gives $a_{12} = a_{22} = 0$, in contradiction to \eqref{eq:kernel}. 

\newTypeProof
%\subsubsection*{\Href{2.V}}\label{p:2.V}
Recall
\[
X = \begin{bmatrix}
0 & 1 & 0 & 0\\
0 & 0 & 0 & \cos\alpha\\
0 & 0 & 0 & \sin\alpha\\
0 & 0 & 0 & 0
\end{bmatrix},\quad 0< \alpha<\pi, \quad H = \begin{bmatrix}
0 & 0 & 1\\
0 & I_2 & 0\\
1 & 0 & 0
\end{bmatrix}.
\]
Equations \eqref{eq:selfadj} and \eqref{eq:kernel} provide
\[
A = \begin{bmatrix}
0 & a_{12} & 0 & a_{14}\\
0 & a_{22} & 0 & a_{12}\\
0 & 0 & 0 & 0\\
0 & 0 & 0 & 0
\end{bmatrix},\quad a_{12},a_{14},a_{22}\in\mathbb{R}.
\]
This gives
\[
AX = \begin{bmatrix}
0 & 0 & 0 & a_{12}\cos\alpha\\
0 & 0 & 0 & a_{22}\cos\alpha\\
0 & 0 & 0 & 0\\
0 & 0 & 0 & 0
\end{bmatrix}\quad\text{ and }\quad XA = \begin{bmatrix}
0 & a_{22} & 0 & a_{12}\\
0 & 0 & 0 & 0\\
0 & 0 & 0 & 0\\
0 & 0 & 0 & 0
\end{bmatrix}.
\]
With \eqref{eq_semi}, this gives $a_{22} = 0$, $a_{12} = a_{12}\cos\alpha$ and thus $a_{12} = 0$. This contradicts \eqref{eq:kernel}.

\stepcounter{subsection}
\newTypeProof
%\subsubsection*{\Href{3.I}}\label{p:3.I}
Recall
\[
X = \begin{bmatrix}
0 & 1 & 0 & 0\\
0 & 0 & z & 0\\
0 & 0 & 0 & 1\\
0 & 0 & 0 & 0
\end{bmatrix},\quad z\in\llbrace -1,1\rrbrace,\quad H = Z_4.
\]
Equations \eqref{eq:selfadj} and \eqref{eq:kernel} give
\[
A = \begin{bmatrix}
0 & a_{12} & a_{13} & a_{14}\\
0 & a_{22} & a_{23} & a_{13}\\
0 & a_{32} & a_{22} & a_{12}\\
0 & 0 & 0 & 0
\end{bmatrix}\in\mathbb{R}^{4,4}.
\]
It then holds
\[
AX = \begin{bmatrix}
0 & 0 & a_{12}z & a_{13}\\
0 & 0 & a_{22}z & a_{23}\\
0 & 0 & a_{32}z & a_{22}\\
0 & 0 & 0 & 0
\end{bmatrix}\quad\text{ and }\quad XA = \begin{bmatrix}
0 & a_{22} & a_{23} & a_{13}\\
0 & a_{32}z & a_{22}z & a_{12}z\\
0 & 0 & 0 & 0\\
0 & 0 & 0 & 0
\end{bmatrix}.
\]
Since \eqref{eq_semi}, we get $a_{22} = a_{32} = 0$ and $a_{23} = za_{12}$. Furthermore,
\[
A^2 = \begin{bmatrix}
0 & 0 & za_{12}^2 & 2a_{12}a_{13}\\
0 & 0 & 0 & za_{12}^2\\
0 & 0 & 0 & 0\\
0 & 0 & 0 & 0
\end{bmatrix} \overset{\eqref{eq:A2}}{=} \begin{bmatrix}
0 & 0 & z & 0\\
0 & 0 & 0 & z\\
0 & 0 & 0 & 0\\
0 & 0 & 0 & 0
\end{bmatrix}
\]
and thus $a_{12}\in\llbrace -1,1\rrbrace$ and $a_{13} = 0$. With
\[
UA = \begin{bmatrix}
0 & a_{12}u_{11} & a_{12}zu_{12} & a_{14}u_{11} + a_{12}u_{13}\\
0 & a_{12}u_{21} & a_{12}zu_{22} & a_{14}u_{21} + a_{12}u_{23}\\
0 & a_{12}u_{31} & a_{12}zu_{32} & a_{14}u_{31} + a_{12}u_{33}\\
0 & a_{12}u_{41} & a_{12}zu_{42} & a_{14}u_{41} + a_{12}u_{43}
\end{bmatrix} \overset{\eqref{eq:UA}}{=} \begin{bmatrix}
0 & 1 & 0 & 0\\
0 & 0 & z & 0\\
0 & 0 & 0 & 1\\
0 & 0 & 0 & 0
\end{bmatrix}
\]
gives us first $u_{11} = u_{22} = u_{33} = a_{12}$, $u_{21} = u_{31} = u_{41} = u_{32} = u_{42} = u_{23} = u_{43} = 0$ and $u_{13} = -a_{14}$. Finally,
\[
U^THU = \begin{bmatrix}
0 & 0 & 0 & a_{12}u_{44}\\
\cdots & 0 & 1 & a_{12}u_{34}\\
\cdots & \cdots & 0 & -a_{14}u_{44} + a_{12}u_{24}\\
\cdots & \cdots & \cdots & 2(u_{14}u_{44} + u_{24}u_{34})
\end{bmatrix} \overset{\eqref{eq:unitary}}{=} \begin{bmatrix}
0 & 0 & 0 & 1\\
0 & 0 & 1 & 0\\
0 & 1 & 0 & 0\\
1 & 0 & 0 & 0
\end{bmatrix}
\]
gives us$u_{44} = a_{12}$, $u_{34} = 0$, $u_{24} = a_{14}$ and $u_{14} = 0$. One can then easily check that
\[
U = \begin{bmatrix}
\epsilon & 0 & -s & 0\\
0 & \epsilon & 0 & s\\
0 & 0 & \epsilon & 0\\
0 & 0 & 0 & \epsilon
\end{bmatrix},\qquad A = \begin{bmatrix}
0 & \epsilon & 0 & s\\
0 & 0 & \epsilon z & 0\\
0 & 0 & 0 &\epsilon\\
0 & 0 & 0 & 0
\end{bmatrix},\quad \epsilon\in\llbrace -1,1\rrbrace, s\in\mathbb{R}.
\]
is indeed a semicommuting $H$-polar decomposition. Since
\[
AU = \begin{bmatrix}0 & 1 & 0 & 2\epsilon s\\
0 & 0 & z & 0\\
0 & 0 & 0 & 1\\
0 & 0 & 0 & 0\end{bmatrix},
\]
it commutes if and only if $s = 0$.

\newTypeProof
%\subsubsection*{\Href{3.II}}\label{p:3.II}
In an analogous manner to \Pref{3.I}.

\newTypeProof
%\subsubsection*{\Href{3.III}}\label{p:3.III}
Recall
\[
X = \begin{bmatrix}
0 & 1 & -r & 0 & s\\
0 & 0 & 1 & r & 0\\
0 & 0 & 0 & -1 & -r\\
0 & 0 & 0 & 0 & -1\\
0 & 0 & 0 & 0 & 0
\end{bmatrix},\quad r,s\in\mathbb{R},\quad H = Z_5.
\]
With \eqref{eq:selfadj} and \eqref{eq:kernel}, we get
\[
A = \begin{bmatrix}
0 & a_{12} & a_{13} & a_{14} & a_{15}\\
0 & a_{22} & a_{23} & a_{24} & a_{14}\\
0 & a_{32} & a_{33} & a_{23} & a_{13}\\
0 & a_{42} & a_{32} & a_{22} & a_{12}\\
0 & 0 & 0 & 0 & 0
\end{bmatrix}\in\mathbb{R}^{5,5}.
\]
It then holds
\[
AX = \begin{bmatrix}
0 & 0 & a_{12} & ra_{12} - a_{13} & -(ra_{13} + a_{14})\\
0 & 0 & a_{22} & ra_{22} - a_{23} & -(ra_{23} + a_{24})\\
0 & 0 & a_{32} & ra_{32} - a_{33} & -(ra_{33} + a_{23})\\
0 & 0 & a_{42} & ra_{42} - a_{32} & -(ra_{43} + a_{22})\\
0 & 0 & 0 & 0 & 0
\end{bmatrix}
\]
and
\[
XA = \begin{bmatrix}
0 & a_{22} - ra_{32} & a_{23} - ra_{33} & a_{24} - ra_{23} & a_{14} -ra_{13}\\
0 & a_{32} + ra_{42} & a_{33} + ra_{32} & a_{23} + ra_{22} & a_{13} + ra_{12}\\
0 & - a_{42} & -a_{32} & -a_{22} & -a_{12}\\
0 & 0 & 0 & 0 & 0\\
0 & 0 & 0 & 0 & 0
\end{bmatrix}.
\]
With \eqref{eq_semi}, this gives $a_{22} = a_{32} = a_{42} = a_{33} = a_{12} = 0$, in contradiction to \eqref{eq:kernel}. 

\newTypeProof
%\subsubsection*{\Href{3.IV}}\label{p:3.IV}
Special case of \Pref{5.I}.

\newTypeProof
%\subsubsection*{\Href{3.V}}\label{p:3.V}
Recall
\[
X = \begin{bmatrix}
0 & 1 & z & 0\\
0 & 0 & 0 & r\\
0 & 0 & 0 & \frac{z}{r}\\
0 & 0 & 0 & 0
\end{bmatrix},\quad z\in\llbrace -1,1\rrbrace, r\in\mathbb{R}, \vert r\vert > 1,\quad H = Z_4.
\]
Equations \eqref{eq:selfadj} and \eqref{eq:kernel} give
\[
A = \begin{bmatrix}
0 & a_{12} & a_{13} & a_{14}\\
0 & a_{22} & a_{23} & a_{13}\\
0 & a_{32} & a_{22} & a_{12}\\
0 & 0 & 0 & 0
\end{bmatrix}.
\]
It then holds
\[
AX = \begin{bmatrix}
0 & 0 & 0 & ra_{12} + \frac{za_{13}}{r}\\
0 & 0 & 0 & ra_{22} + \frac{za_{23}}{r}\\
0 & 0 & 0 & ra_{32} + \frac{za_{22}}{r}\\
0 & 0 & 0 & 0
\end{bmatrix}\quad\text{ and }\quad XA = \begin{bmatrix}
0 & a_{22} + za_{32} & a_{23} + za_{22} & a_{13} + za_{12}\\
0 & 0 & 0 & 0\\
0 & 0 & 0 & 0\\
0 & 0 & 0 & 0
\end{bmatrix}.
\]
With \eqref{eq_semi}, this gives $a_{32} = a_{23} = a_{22} = 0$ and $a_{13} = ra_{12}$. Using
\[
UA = \begin{bmatrix}
\cdots & a_{12}u_{11} & ra_{12}u_{11} & \cdots\\
\cdots & \cdots & \cdots & \cdots\\
\cdots & \cdots & \cdots & \cdots\\
\cdots & \cdots & \cdots & \cdots
\end{bmatrix}\overset{\eqref{eq:UA}}{=}\begin{bmatrix}
0 & 1 & z & 0\\
0 & 0 & 0 & r\\
0 & 0 & 0 & \frac{z}{r}\\
0 & 0 & 0 & 0
\end{bmatrix}
\]
and \eqref{eq:kernel} for $a_{12}\neq 0$, we get $u_{11} = \dfrac{1}{a_{12}}$ and thus
\[
z = ra_{12}\cdot\dfrac{1}{a_{12}} = r.
\] 
This contradicts the assumption $\vert r\vert > 1$. 

\newTypeProof
%\subsubsection*{\Href{3.VI}}\label{p:3.VI}
Recall 
\[
X = \begin{bmatrix}
0 & 1 & 0 & \frac{r^2}{2} & 0\\
0 & 0 & 0 & z & 0\\
0 & 0 & 0 & 0 & r\\
0 & 0 & 0 & 0 & 1\\
0 & 0 & 0 & 0 & 0
\end{bmatrix},\quad z\in\llbrace -1,1\rrbrace, r > 0,\quad H = Z_5.
\]
Equations \eqref{eq:selfadj} and \eqref{eq:kernel} give
\[
A = \begin{bmatrix}
0 & a_{12} & 0 & a_{14} & a_{15}\\
0 & a_{22} & 0 & a_{24} & a_{14}\\
0 & 0 & 0 & 0 & 0\\
0 & a_{42} & 0 & a_{22} & a_{12}\\
0 & 0 & 0 & 0 & 0
\end{bmatrix}\in\mathbb{R}^{5,5}
\]
We obtain
\[
AX = \begin{bmatrix}
0 & 0 & 0 & za_{12} & a_{14}\\
0 & 0 & 0 & za_{22} & a_{24}\\
0 & 0 & 0 & 0 & 0\\
0 & 0 & 0 & za_{42} & a_{22}\\
0 & 0 & 0 & 0 & 0
\end{bmatrix},\quad XA = \begin{bmatrix}
0 & a_{22} + \frac{r^2a_{42}}{2} & 0 & a_{24} + \frac{r^2a_{22}}{2} & a_{14} + \frac{r^2a_{12}}{2}\\
0 & za_{42} & 0 & za_{22} & za_{12}\\
0 & 0 & 0 & 0 & 0\\
0 & 0 & 0 & 0 & 0\\
0 & 0 & 0 & 0 & 0
\end{bmatrix}
\]
and hence with \eqref{eq_semi} $a_{22} = a_{42} = a_{12} = 0$. This then contradicts \eqref{eq:kernel}. 

\newTypeProof
%\subsubsection*{\Href{3.VII}}\label{p:3.VII}
Recall 
\[
X = \begin{bmatrix}
0 & 1 & 0 & 0 & 0 & 0\\
0 & 0 & 1 & 0 & 0 & -\frac{r^2}{2}\\
0 & 0 & 0 & 1 & 0 & 0\\
0 & 0 & 0 & 0 & 0 & 1\\
0 & 0 & 0 & 0 & 0 & r\\
0 & 0 & 0 & 0 & 0 & 0
\end{bmatrix},\quad r > 0,\quad H = \begin{bmatrix}
0 & 0 & 0 & 1\\
0 & Z_3 & 0 & 0\\
0 & 0 & 1 & 0\\
1 & 0 & 0 & 0
\end{bmatrix}.
\]
Equations \eqref{eq:selfadj} and \eqref{eq:kernel} give
\[
A = \begin{bmatrix}
0 & a_{12} & a_{13} & a_{14} & 0 & a_{16}\\
0 & a_{22} & a_{23} & a_{24} & 0 & a_{14}\\
0 & a_{32} & a_{33} & a_{23} & 0 & a_{13}\\
0 & a_{42} & a_{32} & a_{22} & 0 & a_{12}\\
0 & 0 & 0 & 0 & 0 & 0\\
0 & 0 & 0 & 0 & 0 & 0\\
\end{bmatrix}.
\]
We obtain
\[
AX = \begin{bmatrix}
0 & 0 & a_{12} & a_{13} & 0 & a_{14} - \frac{r^2a_{12}}{2}\\
0 & 0 & a_{22} & a_{23} & 0 & \cdots\\
0 & 0 & a_{32}  & a_{33} & 0 & \cdots\\
0 & 0 & a_{42} & a_{32} & 0 & \cdots\\
0 & 0 & 0 & 0 & 0 & 0\\
0 & 0 & 0 & 0 & 0 & 0
\end{bmatrix},\quad XA = \begin{bmatrix}
0 & a_{22} & a_{23} & a_{24} & 0 & a_{14}\\
0 & a_{32} & a_{33} & a_{23} & 0 & \cdots\\
0 & a_{42} & a_{32} & a_{22} & 0 & \cdots\\
0 & 0 & 0 & 0 & 0 & 0\\
0 & 0 & 0 & 0 & 0 & 0\\
0 & 0 & 0 & 0 & 0 & 0
\end{bmatrix}
\]
and thus with \eqref{eq_semi} $a_{22} = a_{32} = a_{42} = a_{12} = 0$ in contradiction to \eqref{eq:kernel}. 

\newTypeProof
%\subsubsection*{\Href{3.VIII}}\label{p:3.VIII}
Recall
\[
X = \begin{bmatrix}
0 & 1 & -2r & 0 & 0 & 0\\
0 & 0 & 1 & r & 0 & -2r^2 + \frac{s^2}{2}\\
0 & 0 & 0 & -1 & 0 & 0\\
0 & 0 & 0 & 0 & 0 & -1\\
0 & 0 & 0 & 0 & 0 & s\\
0 & 0 & 0 & 0 & 0 & 0
\end{bmatrix},\quad r,s\in\mathbb{R}, s > 0,\quad H = \begin{bmatrix}
0 & 0 & 0 & 1\\
0 & Z_3 & 0 & 0\\
0 & 0 & 1 & 0\\
1 & 0 & 0 & 0
\end{bmatrix}.
\]
Equations \eqref{eq:selfadj} and \eqref{eq:kernel} provide
\[
A = \begin{bmatrix}
0 & a_{12} & a_{13} & a_{14} & 0 & a_{16}\\
0 & a_{22} & a_{23} & a_{24} & 0 & a_{14}\\
0 & a_{32} & a_{33} & a_{23} & 0 & a_{13}\\
0 & a_{42} & a_{32} & a_{22} & 0 & a_{12}\\
0 & 0 & 0 & 0 & 0 & 0\\
0 & 0 & 0 & 0 & 0 & 0\\
\end{bmatrix}.
\]
We then obtain
\[
AX = \begin{bmatrix}
0 & 0 & a_{12} & \cdots & 0 & \cdots\\
0 & 0 & a_{22} & ra_{22} - a_{23} & 0 & \cdots\\
0 & 0 & a_{32} & ra_{32} - a_{33} & 0 & \left(-2r^2 + \frac{s^2}{2}\right)a_{32} - a_{23}\\
0 & 0 & a_{42} & \cdots & 0 & \cdots\\
0 & 0 & 0 & 0 & 0 & 0\\
0 & 0 & 0 & 0 & 0 & 0
\end{bmatrix},
\]
\[
XA = \begin{bmatrix}
0 & a_{22} - 2ra_{32} & \cdots & \cdots & 0 & \cdots\\
0 & a_{32}  + ra_{42} & a_{33} + ra_{32} & a_{23} + ra_{22} & 0 & \cdots\\
0 & -a_{42} & -a_{32} & -a_{22} & 0 & -a_{12}\\
0 & 0 & 0 & 0 & 0 & 0\\
0 & 0 & 0 & 0 & 0 & 0\\
0 & 0 & 0 & 0 & 0 & 0
\end{bmatrix}
\]
and thus with \eqref{eq_semi} $a_{42} = a_{32}  = a_{22} = a_{33} = 0$ and hence $a_{12} = a_{23} = 0$. This then contradicts \eqref{eq:kernel}.

\subsubsection*{Types \hyperref[type:3.IX]{3.IX}-\hyperref[type:3.XI]{3.XI}}
\label{p:3.IX}
\label{p:3.X}
\label{p:3.XI}
\label{p:3.XII}
\label{p:5.XII} 
Recall
\[
X = \begin{bmatrix}
0 & X_1\\
0 & 0
\end{bmatrix},\quad X_1\in\mathbb{F}^{2,2}, \det X_1\neq 0,\quad H = \begin{bmatrix}
0 & I_2\\
I_2 & 0
\end{bmatrix}.
\]
We write, using \eqref{eq:selfadj} and \eqref{eq:kernel},
\[
U = \begin{bmatrix}
U_{11} & U_{12}\\
U_{21} & U_{22}
\end{bmatrix}\quad\text{ and }\quad A = \begin{bmatrix}
0 & A_1\\
0 & 0
\end{bmatrix},\quad U_{11},U_{12}, U_{21},U_{22},A_1\in \mathbb{F}^{2,2}
\]
with $A_1^* = A_1$ and $\det A_1\neq 0$. With
\[
UA = \begin{bmatrix}
0 & U_{11}A_1\\
0 & U_{21}A_1
\end{bmatrix}\overset{\eqref{eq:UA}}{=}\begin{bmatrix}
0 & X_1\\
0 & 0
\end{bmatrix},
\]
this gives
\[
U_{21} = 0\quad\text{ and }\quad U_{11} = X_1A_1^ {-1}.
\]
Now,
\[
U^*HU = \begin{bmatrix}
0 & U_{11}^*U_{22}\\
\cdots & U_{12}^*U_{22} + U_{22}^*U_{12}
\end{bmatrix}\overset{\eqref{eq:unitary}}{=}\begin{bmatrix}
0 & I_2\\
I_2 & 0
\end{bmatrix},
\]
and thus
\[
U_{22} = U_{11}^{-*} = X_1^{-*}A_1.
\]
Furthermore, $U_{12}^*U_{22}$ has to be skew-Hermitian. One can easily verify that those conditions are sufficient for a semicommuting $H$-polar decomposition.

Since
\[
AU = \begin{bmatrix}
0 & A_1U_{22}\\
0 & 0
\end{bmatrix},
\]
it commutes if and only if $A_1 = X_1A_1^{-1}X_1^*$.

The details for every special cases can be determined by hand or with the help of a computer.

\begin{comment}
\setcounter{typeC}{11}
\newTypeProof
%\subsubsection*{\Href{3.XII}}\label{p:3XII}
Special case of \Pref{5.XIII}.

\newTypeProof
%\subsubsection*{\Href{3.XIII}}\label{p:3.XIII}
Special case of \Pref{5.XIV}.

\newTypeProof
%\subsubsection*{\Href{3.XIV}}\label{p:3.XIV}
Special case of \Pref{5.XV}.

\newTypeProof
%\subsubsection*{\Href{3.XV}}\label{p:3.XV}
Recall
\[
X = \begin{bmatrix}
0 & 0 & 1 & 0 & 0 & 0\\
0 & 0 & 0 & 1 & r & 0\\
0 & 0 & 0 & 0 & 1 & 0\\
0 & 0 & 0 & 0 & 0 & 1\\
0 & 0 & 0 & 0 & 0 & 0\\
0 & 0 & 0 & 0 & 0 & 0
\end{bmatrix},\quad r > 0,\quad H = \begin{bmatrix}
0 & 0 & I_2\\
0 & I_2 & 0\\
I_2 & 0 & 0
\end{bmatrix}.
\]
\todo{Prinzipiell wie die davor.}

\end{comment}

\stepcounter{subsection}
\newTypeProof
%\subsubsection*{\Href{4.I}}\label{p:4.I}
Recall
\[
X = \begin{bmatrix}
\lambda_1 & 0\\
0 & \lambda_2
\end{bmatrix},\quad \lambda_1,\lambda_2\in\mathbb{C}, \lambda_1\neq \lambda_2, \lambda_1\lambda_2 = 0,\quad H = Z_2.
\]
The following two cases are exhaustive:
\begin{description}
	\item[Case 1: $\lambda_1 = 0, \lambda_2\neq 0$.] Equations \eqref{eq:selfadj} and \eqref{eq:kernel} give
	\[
	A = \begin{bmatrix}
	0 & a_{12}\\
	0 & 0
	\end{bmatrix},\quad a_{12}\in\mathbb{R}\setminus\llbrace 0\rrbrace
	\]
	and thus
	\[
	AX = \begin{bmatrix}
	0 & \lambda_2a_{12}\\
	0 & 0
	\end{bmatrix}\neq 0 = XA
	\]
	in contradiction to \eqref{eq_semi}. 
	
	\item[Case 2: $\lambda_1\neq 0,\lambda_2 = 0$.] Similarly to the above, we get
	\[
	A = \begin{bmatrix}
	0 & 0\\
	a_{21} & 0
	\end{bmatrix},\quad a_{21}\in\mathbb{R}\setminus\llbrace 0\rrbrace
	\]
	and thus
	\[
	AX = \begin{bmatrix}
	0 & 0\\
	\lambda_1a_{21} & 0
	\end{bmatrix} \neq 0 = XA
	\]
	in contradiction to \eqref{eq_semi}. 
\end{description}

\newTypeProof
%\subsubsection*{\Href{4.II}}\label{p:4.II}
Recall
\[
X = \begin{bmatrix}
0 & z\\
0 & 0
\end{bmatrix},\quad z\in\mathbb{C}, \vert z \vert = 1,\quad Z = H_2.
\]
It follows from \eqref{eq:selfadj} and \eqref{eq:kernel}, that
\[
A = \begin{bmatrix}
0 & a_{12}\\
0 & 0
\end{bmatrix},\quad a_{12}\in\mathbb{R}\setminus\llbrace 0\rrbrace.
\]
With
\[
UA = \begin{bmatrix}
0 & a_{12}u_{11}\\
0 & a_{12}u_{21}
\end{bmatrix} \overset{\eqref{eq:UA}}{=} \begin{bmatrix}
0 & z\\
0 & 0
\end{bmatrix},
\]
it follows that $u_{11} = \dfrac{z}{a_{12}}$ and $u_{21} = 0$. This then gives
\[
U^*HU = \begin{bmatrix}
0 & \dfrac{\conj{z}u_{22}}{a_{12}}\\[14pt]
\cdots & 2\Re(u_{12}\conj{u_{22}})
\end{bmatrix}\overset{\eqref{eq:unitary}} = \begin{bmatrix}
0 & 1 \\
1 & 0
\end{bmatrix}
\]
and thus 
\[
u_{22} = \dfrac{a_{12}}{\conj{z}} = za_{12}\quad\text{ and }\quad \Re(u_{12}\conj{u_{22}}) = 0.
\]
This implies $u_{12} = iru_{22} = irza_{12}$ for some $r\in\mathbb{R}$. One can now easily check that
\[
U=\begin{bmatrix}
\frac{z}{s} & irsz\\
0 & sz
\end{bmatrix}\qquad A = \begin{bmatrix}
0 & s \\
0 & 0
\end{bmatrix},\quad r,s\in\mathbb{R}, s\neq 0
\]
is indeed a semicommuting $H$-polar decomposition. Since
\[
AU = \begin{bmatrix}
0 & s^2z\\
0 & 0
\end{bmatrix} = s^2X,
\]
it commutes if and only if $s\in\llbrace -1,1\rrbrace$.

\newTypeProof
%\subsubsection*{\Href{4.III}}\label{p:4.III}
Recall
\[
X = \begin{bmatrix}
0 & z & r\\
0 & 0 & z\\
0 & 0 & 0
\end{bmatrix},\quad r\in\mathbb{R}, z\in\mathbb{C},\; \vert z\vert = 1,\quad H = Z_3.
\]
From equations \eqref{eq:selfadj} and \eqref{eq:kernel}, it follows
\[
A = \begin{bmatrix}
0 & a_{12} & a_{13}\\
0 & a_{22} & \conj{a_{12}}\\
0 & 0 & 0
\end{bmatrix},\quad a_{13},a_{22}\in\mathbb{R}, a_{12}\in\mathbb{C}
\]
and hence
\[
AX = \begin{bmatrix}
0 & 0 & za_{12}\\
0 & 0 & za_{22}\\
0 & 0 & 0
\end{bmatrix}\quad\text{ and }\quad XA = \begin{bmatrix}
0 & za_{22} & z\conj{a_{12}}\\
0 & 0 & 0\\
0 & 0 & 0
\end{bmatrix}.
\]
With \eqref{eq_semi}, this gives $a_{22} = 0$ and $a_{12}\in\mathbb{R}$. Since
\[
A^2 = \begin{bmatrix}
0 & 0 & a_{12}^2\\
0 & 0 & 0\\
0 & 0 & 0
\end{bmatrix} \overset{\eqref{eq:A2}}{=} \begin{bmatrix}
0 & 0 & 1\\
0 & 0 & 0\\
0 & 0 & 0
\end{bmatrix} = X^{[*]}X,
\]
we obtain $a_{12}\in\llbrace -1,1\rrbrace$. Furthermore,
\[
UA = \begin{bmatrix}
0 & a_{12}u_{11} & a_{13}u_{11} + a_{12}u_{12}\\
0 & a_{12}u_{21} & a_{13}u_{21} + a_{12}u_{22}\\
0 & a_{12}u_{31} & a_{13}u_{31} + a_{12}u_{32}
\end{bmatrix} \overset{\eqref{eq:UA}}{=} \begin{bmatrix}
0 & z & r\\
0 & 0 & z\\
0 & 0 & 0
\end{bmatrix}.
\]
This gives first $u_{21} = u_{31} = 0$ and $u_{11} = \dfrac{z}{a_{12}} = za_{12}$ and then $u_{32} = 0$, $u_{22} = \dfrac{z}{a_{12}} = za_{12}$ and $u_{12} = ra_{12} - za_{13}$. We obtain
\[
U^*HU = \begin{bmatrix}
0 & 0 & \conj{z}a_{12}u_{33}\\
\cdots & 1 & (ra_{12} - \conj{z}a_{13})u_{33}  + \conj{z}a_{12}u_{23}\\
\cdots & \cdots & 2\Re(u_{13}\conj{u_{33}}) + \vert u_{23}\vert^2
\end{bmatrix} \overset{\eqref{eq:unitary}}{=} \begin{bmatrix}
0 & 0 & 1\\
0 & 1 & 0\\
1 & 0 & 0
\end{bmatrix}
\]
or
\[
u_{33} = za_{12},\quad u_{23} = z(a_{13} - rza_{12})\quad\text{ and }\quad \Re(u_{13}\conj{z}) = \dfrac{-a_{12}\vert a_{13} - rza_{12}\vert^2}{2}.
\]
One can easily check that
\[
U = \begin{bmatrix}
\epsilon z & \epsilon r - sz & u\\
0 & \epsilon z & sz - \epsilon z^2 r\\
0 & 0  &\epsilon z
\end{bmatrix}\qquad A = \begin{bmatrix}
0 & \epsilon & s\\
0 & 0 & \epsilon\\
0 & 0 & 0
\end{bmatrix},\quad \begin{array}{l}\epsilon \in\llbrace -1,1\rrbrace, s\in\mathbb{R}\\ \Re(u\conj{z}) = \dfrac{-\epsilon\vert s - \epsilon rz\vert^2}{2}\end{array}
\]
is indeed a valid semicommuting $H$-polar decomposition. Since
\[
AU = \begin{bmatrix}
0 & z & 2\epsilon sz - rz^2\\
0 & 0 & z\\
0 & 0 & 0
\end{bmatrix},
\]
it commutes if and only if $2\epsilon sz - rz^2 = r$, i.e.
\[
s = \epsilon r\Re(z).
\]

\newTypeProof
%\subsubsection*{\Href{4.IV}}\label{p:4.IV}
Recall
\[
X = \begin{bmatrix}
0 & 1 & ir\\
0 & 0 & 1\\
0 & 0 & 0
\end{bmatrix},\quad r\in\mathbb{R},\quad H = Z_3.
\]
As for \Pref{4.III}, we get
\[
A = \begin{bmatrix}
0 & \epsilon & s\\
0 & 0 & \epsilon\\
0 & 0 & 0
\end{bmatrix},\quad \epsilon\in\llbrace -1,1\rrbrace, s\in\mathbb{R}.
\]
The equality
\[
UA = \begin{bmatrix}
0 & \epsilon u_{11} & su_{11} + \epsilon u_{12}\\
0 & \epsilon u_{21} & su_{21} + \epsilon u_{22}\\
0 & \epsilon u_{31} & su_{31} + \epsilon u_{32}\\
\end{bmatrix} \overset{\eqref{eq:UA}}{=} \begin{bmatrix}
0 & 1 & ir\\
0 & 0 & 1\\
0 & 0 & 0
\end{bmatrix}
\]
then gives us first $u_{11} = \epsilon$ and $u_{21} = u_{31} = 0$ and then $u_{12} = \epsilon ir -s$, $u_{22} = 0$ and finally $u_{32} = 0$. This then gives
\[
U^*HU = \begin{bmatrix}
0 & 0 & \epsilon u_{33}\\
\cdots & 1 & \epsilon u_{23} - (\epsilon ir + s)u_{33}\\
\cdots & \cdots & 2\Re(u_{13}\conj{u_{33}}) + \vert u_{23}\vert^2
\end{bmatrix} \overset{\eqref{eq:unitary}}{=}\begin{bmatrix}
0 & 0 & 1\\
0 & 1 & 0\\
1 & 0 & 0
\end{bmatrix}
\]
and thus
\[
u_{33} = \epsilon, \quad u_{23} = \epsilon ir +s \quad\text{ and }\quad \Re(u_{13}) = \dfrac{-\epsilon(r^2 + s^2)}{2}.
\]
One can now indeed easily check that
\[
U = \begin{bmatrix}
\epsilon & \epsilon ir - s & u\\
0 & \epsilon & \epsilon ir + s\\
0 & 0 & \epsilon
\end{bmatrix}\qquad A = \begin{bmatrix}
0 & \epsilon & s\\
0 & 0 & \epsilon\\
0& 0 & 0
\end{bmatrix},\quad \begin{array}{l}\epsilon\in\llbrace -1,1\rrbrace, s\in\mathbb{R}\\[8pt]
\Re(u) = \dfrac{-\epsilon(s^2 + r^2)}{2}\end{array}
\]
is a semicommuting $H$-polar decomposition. Since
\[
AU = \begin{bmatrix}
0 & 1 & ir + 2\epsilon s\\
0 & 0 & 1\\
0 & 0 & 0
\end{bmatrix},
\]
it commutes if and only if $s = 0$.

\newTypeProof
%\subsubsection*{\Href{4.V}}\label{p:4.V}
Recall
\[
X = \begin{bmatrix}
0 & \cos\alpha & \sin\alpha & 0\\
0 & 0 & 0 &1\\
0 & 0 & 0 & 0\\
0 &0 & 0 & 0
\end{bmatrix},\quad 0 < \alpha \leq \dfrac{\pi}{2},\quad H = \begin{bmatrix}
0 & 0 & 0 & 1\\
0 & 1 & 0 & 0\\
0 & 0 & 1 & 0\\
1 & 0 & 0 & 0
\end{bmatrix}.
\]
From equations \eqref{eq:selfadj} and \eqref{eq:kernel}, we get
\[
A = \begin{bmatrix}
0 & a_{12}\cos\alpha & a_{12}\sin\alpha & a_{14}\\
0 & a_{32}\cot\alpha\cos\alpha & a_{32}\cos\alpha & \conj{a_{12}}\cos\alpha\\
0 & a_{32}\cos\alpha & a_{32}\sin\alpha & \conj{a_{12}}\sin\alpha\\
0 & 0 & 0 & 0
\end{bmatrix},\quad a_{14},a_{32}\in\mathbb{R}, a_{12}\in\mathbb{C}
\]
and thus
\[
AX = \begin{bmatrix}
0 & 0 & 0 & a_{12}\cos\alpha\\
0 & 0 & 0 & \cdots\\
0 & 0 & 0 & \cdots\\
0 & 0 & 0 & 0
\end{bmatrix}\quad\text{ and }\quad XA = \begin{bmatrix}
0 & \cdots & a_{32} & \conj{a_{12}}\\
0 & 0 & 0 & 0\\
0 & 0 & 0 & 0\\
0 & 0 & 0 & 0
\end{bmatrix}.
\]
With \eqref{eq_semi}, this gives $a_{32} = 0$ and $\conj{a_{12}} = a_{12}\cos\alpha$ and thus $a_{12} = 0$. Furthermore,
\[
A^2 = \begin{bmatrix}
0 & 0 & 0 & \vert a_{12}\vert^2\\
0 & 0 & 0 & 0\\
0 & 0 & 0 & 0\\
0 & 0 & 0 & 0
\end{bmatrix} \overset{\eqref{eq:A2}}{=}\begin{bmatrix}
0 & 0 & 0 & 1\\
0 & 0 & 0 & 0\\
0 & 0 & 0 & 0\\
0 & 0 & 0 & 0
\end{bmatrix}
\]
gives $\vert a_{12}\vert = 1$.

\stepcounter{subsection}

\newTypeProof
Recall
\[
X = \begin{bmatrix}
0 & 1 & 0 & 0\\
0 & 0 & 0 & z\\
0 & 0 & 0 & 0\\
0 & 0 & 0 & 0
\end{bmatrix},\quad z\in\mathbb{C}, \vert z\vert = 1,\quad H = Z_4.
\]
From equations \eqref{eq:selfadj} and \eqref{eq:kernel}, we get
\[
A = \begin{bmatrix}
0 & a_{12} & 0 & a_{14}\\
0 & 0 & 0 & 0\\
0 & a_{42} & 0 & \conj{a_{12}}\\
0 & 0 & 0 & 0
\end{bmatrix},\quad a_{14}, a_{42}\in\mathbb{R}, a_{12}\in\mathbb{C}.
\]
Since
\[
AX = \begin{bmatrix}
0 & 0 & 0 & za_{12}\\
0 & 0 & 0 & 0\\
0 & 0 & 0 & za_{42}\\
0 & 0 & 0 & 0 
\end{bmatrix}\quad\text{ and }\quad XA = 0,
\]
equation \eqref{eq_semi} yields $a_{12} = a_{42} = 0$, in contradiction to \eqref{eq:kernel}.

\newTypeProof

Recall
\[
X = \begin{bmatrix}
0 & 1 & -2ir\Im(z) & 0 & 0 & 0\\
0 & 0 & z & r & 0 & \left(2r^2\Im(z)^2 - \frac{s^2}{2} +it\right)z^2\\
0 & 0 & 0 & z & 0 & 0\\
0 & 0 & 0 & 0 & 0 & z^2\\
0 & 0 & 0 & 0 & 0 & s\\
0 & 0 & 0 & 0 & 0 & 0
\end{bmatrix},\quad \begin{array}{ll}r,s,t\in\mathbb{R}, \\s > 0, \\z\in\mathbb{C}, \\\vert z\vert = 1, \\0 < \arg z<\pi\end{array},\quad H = \begin{bmatrix}
0 & 0 & 0 & 1\\
0 & Z_3 & 0 & 0\\
0 & 0 & 1 & 0\\
1 & 0 & 0 & 0
\end{bmatrix}.
\]
From equations \eqref{eq:selfadj} and \eqref{eq:kernel}, we get
\[
A = \begin{bmatrix}
0 & a_{12} & a_{13} & a_{14} & 0 & a_{16}\\
0 & a_{22} & a_{23} & a_{24} & 0 & \conj{a_{14}}\\
0 & a_{32} & a_{33} & \conj{a_{23}} & 0 & \conj{a_{13}}\\
0 & a_{42} & \conj{a_{32}} & \conj{a_{22}} & 0 & \conj{a_{12}}\\
0 & 0 & 0 & 0 & 0 & 0\\
0 & 0 & 0 & 0 & 0 & 0
\end{bmatrix},\quad a_{16},a_{24},a_{33},a_{42}\in\mathbb{R}, a_{12},a_{13},a_{14},a_{22},a_{23},a_{32}\in\mathbb{C}.
\]
With $l := \left(2r^2\Im(z)^2 - \dfrac{s^2}{2} + it\right)z^2$, we obtain
\[
AX = \begin{bmatrix}
0 & 0 & za_{12} & ra_{12} + za_{13} & 0 & la_{12} + z^2a_{14}\\
0 & 0 & za_{22} & ra_{22} + za_{23} & 0 & la_{22} + z^2a_{24}\\
0 & 0 & za_{32} & ra_{32} + za_{33} & 0 & la_{32} + z^2\conj{a_{23}}\\
0 & 0 & za_{42} & ra_{42} + z\conj{a_{32}} & 0 & la_{42} + z^2\conj{a_{22}}\\
0 & 0 & 0 & 0 & 0 & 0\\
0 & 0 & 0 & 0 & 0 & 0
\end{bmatrix}
\]
and
\[
XA = \begin{bmatrix}
0 & a_{22} -2ir\Im(z)a_{32} & a_{23} - 2ir\Im(z)a_{33} & a_{24} - 2ir\Im(z)\conj{a_{23}} & 0 & \conj{a_{14}} - 2ir\Im(z)\conj{a_{13}}\\
0 & za_{32} + ra_{42} & za_{33} + r\conj{a_{32}} & z\conj{a_{23}} + r\conj{a_{22}} & 0 & z\conj{a_{13}} + r\conj{a_{12}}\\
0 & za_{42} & z\conj{a_{32}} & z\conj{a_{22}} & 0 & z\conj{a_{12}}\\
0 & 0 & 0 & 0 & 0 & 0\\
0 & 0 & 0 & 0 & 0 & 0\\
0 & 0 & 0 & 0 & 0 & 0
\end{bmatrix}.
\]
Equation \eqref{eq_semi} then gives $a_{42} = a_{32} = a_{22} = a_{33} = 0$, $a_{23} = za_{12}$ and $a_{23} = \conj{a_{23}}$. This then gives $a_{23} = u\in\mathbb{R}$ and $a_{12} = \conj{z}u$. This simplifies \eqref{eq_semi} to
\begin{equation}\label{eq:eklig_im_re}
\begin{bmatrix}
r\conj{z}u + za_{13} & l\conj{z}u + z^2a_{14}\\
zu & z^2a_{24}
\end{bmatrix} = \begin{bmatrix}
a_{24} - 2ir\Im(z)u & \conj{a_{14}} - 2ir\Im(z)\conj{a_{13}}\\
zu & z\conj{a_{13}} + rzu
\end{bmatrix}
\end{equation}
which particularly gives
\[
\conj{za_{14}} - 2ir\conj{z}\Im(z)\conj{a_{13}} = \left(2r^2\Im(z) 2 - \dfrac{s^2}{2} + it\right)u + za_{14}
\]
i.e.
\[
\Im(za_{14}) = r\Im(z)\conj{za_{13}} + \left(ir^2\Im(z)^2 - \dfrac{is^2}{4} - \dfrac{t}{2}\right)u.
\]
This means that the right hand side must be a real number and thus
\begin{equation}\label{eq:eklig2}
0 = \Im\left(r\Im(z)\conj{za_{13}} + \left(ir^2\Im(z)^2 - \dfrac{is^2}{4} - \dfrac{t}{2}\right)u\right) = r\Im(z)\Im(\conj{za_{13}}) + r^2u\Im(z)^2 - \dfrac{s^2u}{4}.
\end{equation}
Furthermore, \eqref{eq:eklig_im_re} implies
\[
a_{24} - 2ir\Im(z)u = r\conj{z}u + za_{13}\quad\text{ i.e. }\quad a_{24} = 2ir\Im(z)u + r\conj{z}u + za_{13}
\]
and
\[
z^2a_{24} = z\conj{a_{13}} + rzu\quad\text{ i.e. }\quad a_{24} = \conj{za_{13}} + r\conj{z}u.
\]
Combining these two equations yields
\[
\Im(\conj{za_{13}}) = ru\Im(z).
\]
Using \eqref{eq:eklig2} and $s > 0$, this finally gives $u = 0$ and thus $a_{12} = 0$, in contradiction to \eqref{eq:kernel}.

\newTypeProof
%\subsubsection*{\hyperref[type:5.XVII]{Type 5.XIV}}\label{p:5.XIV}
Recall
\[
X = \begin{bmatrix}
0 & 0 & 1 & 0 & 0\\
0 & 0 & 0 & 1 & 0\\
0 & 0 & 0 & z & 0\\
0 & 0 & 0 & 0 & 0\\
0 & 0 & 0 & 0 & 0
\end{bmatrix},\quad z\in\mathbb{C}, \vert z\vert = 1,\quad H = \begin{bmatrix}
0 & 0 & I_2\\
0 & 1 & 0\\
I_2 & 0 & 0
\end{bmatrix}.
\]\begin{comment}
It follows from \eqref{eq:selfadj} and \eqref{eq:kernel}, that
\[
A = \begin{bmatrix}
0 & 0 & a_{13} & a_{14} & 0\\
0 & 0 & 0 & 0 & 0\\
0 & 0 & a_{33} & \conj{a_{13}} & 0\\
0 & 0 & 0 & 0 & 0\\
0 & 0 & 0 & 0 & 0
\end{bmatrix},\quad a_{14},a_{33}\in\mathbb{R}.
\]
This already implies that there does not exist any commuting $H$-polar decomposition, since the second line of $AU$ is zero for all $U\in\mathbb{C}^{5,5}$ and thus $AU \neq X$.
\end{comment}
It follows directly from \ref{theo:commuting_factors} that no commuting $H$-polar decomposition can exist. On the other hand, it is easily verified that the given decomposition is a semicommuting $H$-polar decomposition. \begin{comment}
Now it is
\[
AX = \begin{bmatrix}
0 & 0 & 0 & za_{13} & 0\\
0 & 0 & 0 & 0 & 0\\
0 & 0 & 0 & za_{33} & 0\\
0 & 0 & 0 & 0 & 0\\
0 & 0 & 0 & 0 & 0
\end{bmatrix},\quad XA = \begin{bmatrix}
0 & 0 & a_{33} & \conj{a_{13}} & 0\\
0 & 0 & 0 & 0 & 0\\
0 & 0 & 0 & 0 & 0\\
0 & 0 & 0 & 0 & 0\\
0 & 0 & 0 & 0 & 0
\end{bmatrix}
\]
and thus we obtain with \eqref{eq_semi}, that
\[
A = \begin{bmatrix}
0 & 0 & a_{13} & a_{14} & 0\\
0 & 0 & 0 & 0 & 0\\
0 & 0 & 0 & za_{13} & 0\\
0 & 0 & 0 & 0 & 0\\
0 & 0 & 0 & 0 & 0
\end{bmatrix}.
\]
Since
\[
X^{[*]}X = \begin{bmatrix}
0 & 0 & 0 & 1 & 0\\
0 & 0 & 0 & 0 & 0\\
0 & 0 & 0 & 0 & 0\\
0 & 0 & 0 & 0 & 0\\
0 & 0 & 0 & 0 & 0
\end{bmatrix}\quad\text{ and }\quad A^2 = \begin{bmatrix}
0 & 0 & 0 & za_{13}^2 & 0\\
0 & 0 & 0 & 0 & 0\\
0 & 0 & 0 & 0 & 0\\
0 & 0 & 0 & 0 & 0\\
0 & 0 & 0 & 0 & 0
\end{bmatrix},
\]
it follows from \eqref{eq:A2} that $za_{13}^2$.\end{comment}

\newTypeProof
%\subsubsection*{\hyperref[type:5.XVII]{Type 5.XVII}}\label{p:5.XVII}
Recall
\[
X = \begin{bmatrix}
0 & 0 & 1 & 0 & 0 & 0 & 0\\
0 & 0 & 0 & 1 & 0 & 0 & 0\\
0 & 0 & 0 & 0 & 0 & -z_1\conj{z_2}\cos\alpha & \sin\alpha\cos\beta\\
0 & 0 & 0 & 0 & 0 & z_1\sin\alpha & z_2\cos\alpha\cos\beta\\
0 & 0 & 0 & 0 & 0 & 0 & \sin\beta\\
0 & 0 & 0 & 0 & 0 & 0 & 0\\
0 & 0 & 0 & 0 & 0 & 0 & 0
\end{bmatrix},\quad \begin{array}{ll}z_1,z_2\in\mathbb{C}, \\\vert z_1\vert = \vert z_2\vert = 1\\ 0<\alpha,\beta\leq \dfrac{\pi}{2}\end{array},\quad H = \begin{bmatrix}
0 & 0 & I_2\\
0 & I_3 & 0\\
I_2 & 0 & 0
\end{bmatrix}.
\]
It follows from \eqref{eq:selfadj} and \eqref{eq:kernel}, that
\[
A = \begin{bmatrix}
0 & 0 & a_{13} & a_{14} & 0 & a_{16} & a_{17}\\
0 & 0 & a_{23} & a_{24} & 0 & \conj{a_{17}} & a_{27}\\
0 & 0 & a_{33} & a_{34} & 0 & \conj{a_{13}} & \conj{a_{23}}\\
0 & 0 & \conj{a_{34}} & a_{44} & 0 & \conj{a_{14}} & \conj{a_{24}}\\
0 & 0 & 0 & 0 & 0 & 0 & 0\\
0 & 0 & 0 & 0 & 0 & 0 & 0\\
0 & 0 & 0 & 0 & 0 & 0 & 0
\end{bmatrix},\quad \begin{array}{l}a_{16},a_{27},a_{33},a_{44}\in\mathbb{R},\\ a_{13},a_{14},a_{17}, a_{23}, a_{24}, a_{34}\in\mathbb{C}\end{array}.
\]
It is
\[
AX = \begin{bmatrix}
0 & 0 & 0 & 0 & 0 & -z_1\conj{z_2}a_{13}\cos\alpha + a_{14}z_1\sin\alpha & a_{13}\sin\alpha\cos\beta + a_{14}z_2\cos\alpha\cos\beta\\
0 & 0 & 0 & 0 & 0 & -z_1\conj{z_2}a_{23}\cos\alpha + a_{24}z_1\sin\alpha & a_{23}\sin\alpha\cos\beta + a_{24}z_2\cos\alpha\cos\beta\\
0 & 0 & 0 & 0 & 0 & \cdots & \cdots\\
0 & 0 & 0 & 0 & 0 & \cdots & \cdots\\
0 & 0 & 0 & 0 & 0 & 0 & 0\\
0 & 0 & 0 & 0 & 0 & 0 & 0\\
0 & 0 & 0 & 0 & 0 & 0 & 0
\end{bmatrix},
\]
\[
XA = \begin{bmatrix}
0 & 0 & \cdots & \cdots & 0 & \conj{a_{13}} & \conj{a_{23}}\\
0 & 0 & \cdots & \cdots & 0 & \conj{a_{14}} & \conj{a_{24}}\\
0 & 0 & 0 & 0 & 0 & 0 & 0\\
0 & 0 & 0 & 0 & 0 & 0 & 0\\
0 & 0 & 0 & 0 & 0 & 0 & 0\\
0 & 0 & 0 & 0 & 0 & 0 & 0\\
0 & 0 & 0 & 0 & 0 & 0 & 0
\end{bmatrix}
\]
and thus with \eqref{eq_semi}
\[
\llbrace\begin{array}{rcl}
-z_1\conj{z_2}\cos\alpha\cdot a_{13} + z_1\sin\alpha\cdot a_{14} &=& \conj{a_{13}}\\
\sin\alpha\cos\beta\cdot a_{13} + z_2\cos\alpha\cos\beta\cdot a_{14} &=& \conj{a_{23}}\\
-z_1\conj{z_2}\cos\alpha\cdot a_{23} + z_1\sin\alpha\cdot a_{24} &=& \conj{a_{14}}\\
\sin\alpha\cos\beta\cdot a_{23} + z_2\cos\alpha\cos\beta\cdot a_{24} &=& \conj{a_{24}}
\end{array}\right..
\]
This system admits* the unique solution $a_{13} = a_{14} = a_{23} = a_{24} = 0$ implying $\ker A \neq \ker X$ in contradiction too \eqref{eq:kernel}. 

\newTypeProof
%\subsubsection*{\hyperref[type:5.XVIII]{Type 5.XVIII}}\label{p:5.XVIII}
Recall
\[
X = \begin{bmatrix}
0 & 0 & 1 & 0 & 0 & 0 & 0 & 0\\
0 & 0 & 0 & 1 & 0 & 0 & 0 & 0\\
0 & 0 & 0 & 0 & 0 & 0 & -z_1\conj{z_2}\sin\alpha\cos\beta & \cos\alpha\cos\gamma\\
0 & 0 & 0 & 0 & 0 & 0 & z_1\cos\alpha\cos\beta & z_2\sin\alpha\cos\gamma\\
0 & 0 & 0 & 0 & 0 & 0 & \sin\beta & 0\\
0 & 0 & 0 & 0 & 0 & 0 & 0 & \sin\gamma\\
0 & 0 & 0 & 0 & 0 & 0 & 0 & 0\\
0 & 0 & 0 & 0 & 0 & 0 & 0 & 0
\end{bmatrix},\quad\begin{array}{ll}
z_1,z_2\in\mathbb{C}\\
\vert z_1\vert = \vert z_2\vert = 1\\[6pt]
0\leq \alpha < \dfrac{\pi}{2}\\
0 < \beta < \gamma\leq \dfrac{\pi}{2}
\end{array},\quad H = \begin{bmatrix}
0 & 0 & I_2\\
0 & I_4 & 0\\
I_2 & 0 & 0
\end{bmatrix}.
\]
It follows from \eqref{eq:selfadj} and \eqref{eq:kernel}, that
\[
A = \begin{bmatrix}
0 & 0 & a_{13} & a_{14} & 0 & 0 & a_{17} & a_{18}\\
0 & 0 & a_{23} & a_{24} & 0 & 0 & \conj{a_{18}} & a_{28}\\
0 & 0 & a_{33} & a_{34} & 0 & 0 & \conj{a_{13}} & \conj{a_{23}}\\
0 & 0 & \conj{a_{34}} & a_{44} & 0 & 0 & \conj{a_{14}} & \conj{a_{24}}\\
0 & 0 & 0 & 0 & 0 & 0 & 0 & 0\\
0 & 0 & 0 & 0 & 0 & 0 & 0 & 0\\
0 & 0 & 0 & 0 & 0 & 0 & 0 & 0\\
0 & 0 & 0 & 0 & 0 & 0 & 0 & 0
\end{bmatrix},\quad \begin{array}{l}a_{17},a_{28},a_{33},a_{44}\in\mathbb{R},\\ a_{13},a_{14},a_{18}, a_{23}, a_{24}, a_{34}\in\mathbb{C}\end{array}.
\]
It is
\[
AX = \begin{bmatrix}
0 & 0 & 0 & 0 & 0 & 0 & -z_1\conj{z_2}a_{13}\sin\alpha\cos\beta + a_{14}z_1\cos\alpha\cos\beta & a_{13}\cos\alpha\cos\gamma + a_{14}z_2\sin\alpha\cos\gamma\\
0 & 0 & 0 & 0 & 0 & 0 & -z_1\conj{z_2}a_{23}\sin\alpha\cos\beta + a_{24}z_1\cos\alpha \cos\beta& a_{23}\cos\alpha\cos\gamma + a_{24}z_2\sin\alpha\cos\gamma\\
0 & 0 & 0 & 0 & 0 & 0 & \cdots & \cdots\\
0 & 0 & 0 & 0 & 0 & 0 & \cdots & \cdots\\
0 & 0 & 0 & 0 & 0 & 0 & 0 & 0\\
0 & 0 & 0 & 0 & 0 & 0 & 0 & 0\\
0 & 0 & 0 & 0 & 0 & 0 & 0 & 0\\
0 & 0 & 0 & 0 & 0 & 0 & 0 & 0
\end{bmatrix},
\]
\[
XA = \begin{bmatrix}
0 & 0 & \cdots & \cdots & 0 & 0 & \conj{a_{13}} & \conj{a_{23}}\\
0 & 0 & \cdots & \cdots & 0 & 0 & \conj{a_{14}} & \conj{a_{24}}\\
0 & 0 & 0 & 0 & 0 & 0 & 0 & 0\\
0 & 0 & 0 & 0 & 0 & 0 & 0 & 0\\
0 & 0 & 0 & 0 & 0 & 0 & 0 & 0\\
0 & 0 & 0 & 0 & 0 & 0 & 0 & 0\\
0 & 0 & 0 & 0 & 0 & 0 & 0 & 0\\
0 & 0 & 0 & 0 & 0 & 0 & 0 & 0
\end{bmatrix}
\]
and thus with \eqref{eq_semi}
\[
\llbrace\begin{array}{rcl}
-z_1\conj{z_2}\sin\alpha\cos\beta\cdot a_{13} + z_1\cos\alpha\cos\beta\cdot a_{14} &=& \conj{a_{13}}\\
\cos\alpha\cos\gamma\cdot a_{13} + z_2\sin\alpha\cos\gamma\cdot a_{14} &=& \conj{a_{23}}\\
-z_1\conj{z_2}\sin\alpha\cos\beta\cdot a_{23} + z_1\cos\alpha\cos\beta\cdot a_{24} &=& \conj{a_{14}}\\
\cos\alpha\cos\gamma\cdot a_{23} + z_2\sin\alpha\cos\gamma\cdot a_{24} &=& \conj{a_{24}}
\end{array}\right..
\]
This system admits* the unique solution $a_{13} = a_{14} = a_{23} = a_{24} = 0$ implying $\ker A \neq \ker X$ in contradiction too \eqref{eq:kernel}.

\end{appendices}

\end{document}